\documentclass[11pt]{article}
\usepackage[T1]{fontenc}

\usepackage{authblk}
\usepackage{fullpage}
\usepackage{algorithmic}
\usepackage{algorithm}
\usepackage{epsfig}
\usepackage{graphicx}
\usepackage{latexsym}
\usepackage{amsmath}
\usepackage{amsfonts}
\usepackage{multirow}
\usepackage{amssymb}
\usepackage[sort,numbers]{natbib}
\usepackage{tikz}

\usetikzlibrary{shapes,arrows,shadows}
\usepackage{amsmath,bm,times}

\usepackage{mathrsfs}
\usepackage{pifont}
\usepackage{yhmath}
\usepackage{bbm}
\usepackage{amsthm}
\usepackage{booktabs}
\usepackage{stmaryrd}
\usepackage{wasysym}
\usepackage{bbm}
\usepackage{array}
\usepackage{enumerate}
\usepackage{dsfont}

\numberwithin{equation}{section}
\numberwithin{figure}{section}

\newtheorem{theorem}{Theorem}[section]
\newtheorem{lemma}[theorem]{Lemma}

\newtheorem{corollary}[theorem]{Corollary}

\newtheorem{definition}[theorem]{Definition}
\newtheorem{remark}[theorem]{Remark}

\theoremstyle{remark}

\makeatletter
\newcommand\figcaption{\def\@captype{figure}\caption}
\newcommand\tabcaption{\def\@captype{table}\caption}
\makeatother

\DeclareMathAlphabet{\mathpzc}{OT1}{pzc}{m}{it}

\begin{document}
	\newcounter{my}
	\newenvironment{mylabel}
	{
		\begin{list}{(\roman{my})}{
				\setlength{\parsep}{-1mm}
				\setlength{\labelwidth}{8mm}
				\usecounter{my}}
		}{\end{list}}
	
	\newcounter{my2}
	\newenvironment{mylabel2}
	{
		\begin{list}{(\alph{my2})}{
				\setlength{\parsep}{-0mm} \setlength{\labelwidth}{8mm}
				\setlength{\leftmargin}{3mm}
				\usecounter{my2}}
		}{\end{list}}
	
	\newcounter{my3}
	\newenvironment{mylabel3}
	{
		\begin{list}{(\alph{my3})}{
				\setlength{\parsep}{-1mm}
				\setlength{\labelwidth}{8mm}
				\setlength{\leftmargin}{10mm}
				\usecounter{my3}}
		}{\end{list}}

	\title{\bf The log-Sobolev inequality and correlation functions for the renormalization of 1D Ising model }
	\author[a]{Cui Kaiyuan
	}
	\author[b]{Gong Fuzhou}
	\affil[a]{Institute of Applied Mathematics, Academy of Mathematics and Systems Science, Chinese Academy of Sciences, Beijing 100190, China,({cuiky@amss.ac.cn})}
	\affil[b]{Institute of Applied Mathematics, Academy of Mathematics and Systems Science, Chinese Academy of Sciences, Beijing 100190, China,({ fzgong@amt.ac.cn})}
	\renewcommand*{\Affilfont}{\small\it}
	\renewcommand\Authands{and}
	\date{}
	\maketitle

	\vspace{-5em}
	
	\begin{center}\large

	\end{center}

	
	\begin{abstract}
		The renormalization group (RG) method is an important tool for studying critical phenomena. In this paper, we employ stochastic analysis techniques to investigate the stochastic partial differential equation (SPDE) derived by regularizing and continuousizing the discrete stochastic equation, which is a variant of stochastic quantization equation of the one dimensional (1D) Ising model. Firstly, we give the regularity estimates for the solution to SPDE. Secondly, we prove the Clark-Ocone-Haussmann formula and derive the log-Sobolev inequality up to the terminal time $T$, as well as obtain a priori form of the renormalization relation. Finally, we verify the correctness of the renormalization procedure based on the partition function, and prove that the two point correlation functions of SPDE on lattices converge to the two point correlation functions of the 1D Ising model at the stable fixed point of the RG transformation as $T\rightarrow +\infty$.\\
		
	\end{abstract}
	
	\section{Introduction}\label{sec:1}
	The renormalization group (RG) method, formulated by Wilson in 1971, has had a profound influence on both mathematics and physics. Specifically, the fundamental idea behind the real space RG is the 'coarse-graining' procedure for a discrete spin system, which  connects the Gibbs measures of the system at different scales and can be described by a discrete time Markov chain from a probabilistic perspective, see \cite{Morandi2001document} for example. Motivated by the stochastic quantization proposed by Parisi and Wu \cite{PARISI1981document}, we reconstructed the renormalization  procedure using a stochastic equation. Based on the renormalization relation in the one dimensional (1D) Ising model, we derived the local Poincaré inequalities for the two point functions of the Poisson continuous version of the discrete stochastic equation in \cite{Cui2022document}. We also selected the relative entropies of the spin functions in \cite{Cui2024document} as observables, and proved that the renormalization equations grounded in these observables hold, based on the standard renormalization procedure in physics, for example see \cite{Zinn2007document}.
	
	Different from \cite{Cui2022document} and \cite{Cui2024document}, in this paper, we regularize the variant of the stochastic quantization equation by mollifying its sign function, and take the continuity limit in space as well as time to obtain a stochastic partial differential equation (SPDE). For the SPDE, we establish the log-Sobolev inequality and the Poincar\'e inequality up to the terminal time $T$. To prove them, we first derive a Clark-Ocone-Haussmann formula, following a strategy analogous to that used in \cite{Capitaine1997document, Fang1999document, Gong1998document} for proving the log-Sobolev inequalities on path and loop spaces. Guided by the constant in the log-Sobolev inequality which depends on the regularized and coupling parameters of the system as well as the terminal time $T$, and the principle that the constant does not blow up as $T\rightarrow+\infty$, we provide a priori choice of the renormalization relation. Moreover, we prove the correctness of the renormalization procedure, based on the invariance of the partition function and the obtained renormalization relation. Finally, we study the long time behavior of the two point
	correlation functions of SPDE on lattices, and prove that they converge to the two point correlation functions of the 1D Ising model at the stable fixed point of the RG transformation as $T\rightarrow +\infty$ by the Poincar\'e inequalities.
	
	
    One of the key points in investigating statistical mechanics models is proving the related log-Sobolev inequalities. Recently, there are new developments in deriving log-Sobolev inequalities for the models in statistical mechanics and  Euclidean quantum fields via alternative methods. For example, by means of log-Sobolev inequalities, Shen, Zhu and Zhu in \cite{Shen2021document,Shen2024document} proved that the correlations of a large class of observables decay exponentially for lattice Yang-Mills and Yang-Mills-Higgs fields at strong coupling regime. Bauerschmidt and coauthor in \cite{Bauerschmidt2021document} proposed a new approach based on the Pochinski flow in renormalization theory to establish the log-Sobolev inequality for the continuum sine-Gordon model. With the same approach, Bauerschmidt and coauthor in \cite{Bauerschmidt20241document} derived the log-Sobolev inequalities for the $\phi^{4}_{2}$ and $\phi^{4}_{3}$ in the entire high temperature phases, based on the known renormalization relations which are dimension-dependent divergent counterterms and a correlation inequality proven by Ding, Song and Sun \cite{Ding2023document}. Different from the above work, the result in this paper shows that a good estimate for the constant in the log-Sobolev inequality up to the terminal time $T$ can be directly used to determine the renormalization relation. Furthermore, it is worth noting that such an approach of selecting renormalization relation based on the log-Sobolev inequality has the potential to be extended to more complex models.
	
 This paper is organized as follows. In Section \ref{sec:2}, we heuristically derive a SPDE by regularizing and continuousizing the discrete stochastic equation formulated in \cite{Cui2022document,Cui2024document}. Besides, for the reader's convenience, we present some elementary concepts and estimates. In Section \ref{sec:3}, we establish the regularity estimates for the solution to SPDE. In Section \ref{sec:4}, we derive the Clark-Ocone-Haussmann formula and the log-Sobolev inequality up to the terminal time $T$ based on the integration by parts formula. In Section \ref{sec:5}, we prove the correctness of the renormalization procedure based on the partition function and study the long time behavior of the correlation functions. 
	
\section{Preliminaries}\label{sec:2}
\subsection{Model}\label{sec:2.1}
	As introduced in Section \ref{sec:1}, inspired by the idea of stochastic quantization proposed by Parisi and Wu \cite{PARISI1981document}, we reconstructed the renormalization  procedure in  \cite{Cui2022document} and \cite{Cui2024document} using a discrete stochastic equation as follows
	\begin{align}\label{eq.18}
		\phi(y_{i},t+1)&={\rm sgn}(K\tilde{\Delta }\phi(y_{i},t)-(\gamma-1-2K)\phi(y_{i},t)+\xi(y_{i},t)),
	\end{align}
	where $K,\gamma$ are coupling parameters of the equation, $\{\xi(y_{i},t)\}_{i\in\mathbb{Z},t\in\mathbb{Z}^{+}}$ are standard Gaussian random variables that are independent of each other for $t\in\mathbb{Z}^{+},y_{i}=i\in\mathbb{Z}$, and
	\begin{align*}
		\tilde{\Delta }\phi(y_{i},t):=\phi(y_{i}-1,t)+\phi(y_{i}+1,t)-2\phi(y_{i},t).
	\end{align*}
	In this section, we will heuristically regularize and continuousize the above equation to obtain a SPDE. Firstly, approximating Equation~\eqref{eq.18}~by
	\begin{align*}
		\phi(y_{i},t+1)=g_{\varepsilon}(K\tilde{\Delta }\phi(y_{i},t)-(\gamma-1-2K)\phi(y_{i},t)+B(y_{i},t+1)-B(y_{i},t)),
	\end{align*}
	where $\{B(y_{i},t)\}_{y_{i}\in\mathbb{Z}}$ is a family of standard Brownian motions that are independent of each other for different $y_{i}$, and $g_{\varepsilon}(x)$ is a continuous, increasing function converging to ${\rm sgn}(x)$ as the regularized parameter 
	$\varepsilon\rightarrow 0$. Particularly, 
	we take $g_{\varepsilon}(x)$ as follows
	\begin{align*}
		g_{\varepsilon}(x)&=\left\{
		\begin{array}{ll}
			\varepsilon x+1-\varepsilon^{2} &x\geq \varepsilon>0,\\
			\frac{x}{\varepsilon} &-\varepsilon <x< \varepsilon,\\
			\varepsilon x-1+\varepsilon^{2} &x\leq -\varepsilon<0.
		\end{array}\right.
	\end{align*}
	Obviously, $g_{\varepsilon}(x)$ is increasing and
	\begin{align*}
		g^{-1}_{\varepsilon}(x)&=\left\{
		\begin{array}{ll}
			\frac{x}{\varepsilon} +\varepsilon-\frac{1}{\varepsilon}  &x\geq 1,\\
			\varepsilon x &-1 <x< 1,\\
			\frac{x}{\varepsilon} -\varepsilon+\frac{1}{\varepsilon}  &x\leq -1.
		\end{array}\right.
	\end{align*}
    Then we get 
	$$g^{-1}_{\varepsilon}(\phi(y_{i},t+1))=K\tilde{\Delta } \phi(y_{i},t)-(\gamma-1-2K)\phi(y_{i},t)+B(y_{i},t+1)-B(y_{i},t),$$
	and the above equation can be rewritten as
	\begin{align*}
		\phi(y_{i},t+1)-\phi(y_{i},t)&=K\tilde{\Delta } \phi(y_{i},t)-(\gamma-2K)\phi(y_{i},t)+\phi(y_{i},t+1)\\
		&-g^{-1}_{\varepsilon}(\phi(y_{i},t+1))+B(y_{i},t+1)-B(y_{i},t).
	\end{align*}
	Replacing the time step with $h$, we get
	\begin{align*}
		\phi(y_{i},t+h)-\phi(y_{i},t)&=K\tilde{\Delta } \phi(y_{i},t)-(\gamma-2K)\phi(y_{i},t)+\phi(y_{i},t+h)\\
		&-g^{-1}_{\varepsilon}(\phi(y_{i},t+h))+B(y_{i},t+h)-B(y_{i},t).
	\end{align*}
Given the standard mollifier
	\begin{equation*}
		\tilde{J}(x)=\left\{\begin{array}{ll}
			\hat{J} e^{-\frac{1}{1-x^{2}}} & |x|< 1,  \\
			0 & |x|\geq 1,
		\end{array}\right.
	\end{equation*}
	where $\hat{J}>0$ is a normalization constant such that $\int_{\mathbb{R}}\tilde{J}(x)dx=1$.
	Let $U_{\varepsilon,\delta}(x):=-\tilde{J}^{\delta}\ast g^{-1}_{\varepsilon}(x)$, where   $\tilde{J}^{\delta}(x):=\frac{1}{\delta}\tilde{J}(\frac{x}{\delta})$ and $\delta>0$. Let 
	\begin{align*}
			\delta_{1}(x):=&\frac{\hat{J}}{\delta}\int_{-\delta}^{x-1}e^{-\frac{1}{1-(y/\delta)^{2}}}dy, -\delta \leq x-1\leq \delta,\\
			\delta_{2}(x):=&\frac{\hat{J}}{\delta}\int_{-\delta}^{x+1}e^{-\frac{1}{1-(y/\delta)^{2}}}dy, -\delta \leq x+1\leq \delta,
	\end{align*}
	and we can calculate explicitly that 
	\begin{align*}
		U^{\prime}_{\varepsilon,\delta}(x)&=\left\{
		\begin{array}{ll}
			-\frac{1}{\varepsilon} &x\geq 1+\delta,\\
			 -\varepsilon(1-\delta_{1}(x))-\frac{1}{\varepsilon}\delta_{1}(x) &1-\delta <x< 1+\delta ,\\
			-\varepsilon &-1+\delta \leq x\leq  1-\delta ,\\
			-\varepsilon  \delta_{2}(x)-\frac{1}{\varepsilon}(1-\delta_{2}(x)) &-1-\delta <x< -1+\delta ,\\
			-\frac{1}{\varepsilon} &x\leq -1-\delta,
		\end{array}\right.
	\end{align*}
	which means $-1/\varepsilon\leq U^{\prime}_{\varepsilon,\delta}(x)\leq -\varepsilon$ for $\varepsilon\leq 1$.	In this manner, Equation~\eqref{eq.18}~can be regularized as 
	\begin{align*}
		\phi(y_{i},t+h)-\phi(y_{i},t)&=K\tilde{\Delta } \phi(y_{i},t)-(\gamma-2K)\phi(y_{i},t)+\phi(y_{i},t+h)\\
		&+U_{\varepsilon,\delta}(\phi(y_{i},t+h))+B(y_{i},t+h)-B(y_{i},t).
	\end{align*}
	 Furthermore, taking the time step $h\rightarrow 0$ and the spatial continuum limit $\tilde{\Delta }\rightarrow\Delta $, we get the regularized and continuousized version of Equation~\eqref{eq.18}
	\begin{align*}
		\frac{\partial}{\partial t}\phi(y,t)=K\Delta \phi(y,t)-(\gamma-2K)\phi(y,t)+(\phi(y,t)+U_{\varepsilon,\delta}(\phi(y,t))+\dot{\mathcal{W}}^{\rho}(y,t),
	\end{align*}	
	where $\dot{\mathcal{W}}^{\rho}(y,t)$ is a colored noise that approximating the space-time white noise, and the correlation function between spatial points $y_{i}$ and $y_{j}$ converges to the Dirac delta function $\delta(y_{i}-y_{j})$ as $\rho\rightarrow0$.	
	
	Let us consider the model with periodic boundary conditions, and the equation will be defined on a bounded open interval $(-M,M)\subset\mathbb{R}$. Taking the kernel function as follows 
		\begin{align*}
		\rho_{M}(x,y)&:=\frac{1}{\sqrt{8M\pi\rho}}\sum_{n=-\infty}^{+\infty} \left[e^{-\frac{(4Mn+y-x)^{2}}{16M^{2}\rho}}-e^{-\frac{(4Mn+y+x)^{2}}{16M^{2}\rho}}\right], 		
	\end{align*}
	where $\rho>0$ and $x,y\in[-M,M]$. Obviously, the kernel function $\rho_{M}(x,y)$ converges to the Dirac delta function $\delta(x-y)$ as $\rho\rightarrow0$. Denote by $W^{\rho_{M}}(y,t)$ the process satisfying
\begin{align*}
	dW^{\rho_{M}}(y,t)&:=\int_{-M}^{M}\rho_{M}(y,z)W_{M}(dz,dt),
\end{align*}
where $W_{M}(z,t)$ is the Brownian sheet on $[-M,M]\times[0,T]$. Let $C_{c}^{\infty}((-M,M))$ denote the set of all infinitely differentiable real-valued functions on $(-M,M)$ with compact support. For $\hat{u}\in C_{c}^{\infty}((-M,M))$, we consider the following semi-linear SPDE defined on the bounded open interval $(-M,M)\subset\mathbb{R}$
	\begin{equation}\label{eq:2.3}
		\left\{
		\begin{aligned}
			&d Y(y,t)=K\Delta Y(y,t)dt+2K Y(y,t)dt+b(Y(y,t))dt+ dW^{\rho_{M}}(y,t),\ t> 0,\\
			&Y(y,0)=\hat{u}(y),\forall{y\in (-M,M)},\\
			&Y(-M,t)=Y(M,t)=0,
		\end{aligned}
		\right.
	\end{equation}
	where 
	\begin{align}\label{eq:drift}
		b(x)&:=U_{\varepsilon,\delta}(x)+(1-\gamma)x,		
	\end{align}
	and the parameter $\gamma>1$ is arbitrary, as noted in \cite{Cui2022document}.	For $0\leq x,y\leq 1$ and $\rho>0$, let $W^{\rho}(x,t)$ denote the process satisfying
	\begin{align*}
		dW^{\rho}(x,t)&:=\int_{0}^{1}\tilde{\rho}(x,y)W(dy,dt),
	\end{align*}
	where $W(y,t)$ is the Brownian sheet on $[0,1]\times[0,T]$ and
	\begin{align}\label{eq:2.5}
		\tilde{\rho}(x,y):=\frac{1}{\sqrt{4\pi\rho}}\sum_{n=-\infty}^{+\infty}\left[e^{-\frac{(2n+y-x)^{2}}{4\rho}}-e^{-\frac{(2n+y+x)^{2}}{4\rho}}\right], x,y\in[0,1].
	\end{align}
	Then
	\begin{align*}
		E[W^{\rho}(x,t)W^{\rho}(y,t)]=&E\left[\int_{0}^{t}\int_{0}^{1}\tilde{\rho}(x,z)W(dz,ds)\int_{0}^{t}\int_{0}^{1}\tilde{\rho}(y,z)W(dz,ds)\right]\\
		=&t\int_{0}^{1}\tilde{\rho}(x,z)\tilde{\rho}(z,y)dz=\mathcal{K}(x,y)t,
	\end{align*}
	where
		\begin{align}\label{eq:noisecova}
		\mathcal{K}(x,y):=\int_{0}^{1}\tilde{\rho}(x,z)\tilde{\rho}(z,y)dz.
	\end{align}
	Fixing the terminal time $T>0$, under the spatial scaling transformation
	\begin{align*}
		x=\frac{y+M}{2M}\in(0,1),
	\end{align*}
	the equation $X(x,t)=Y(y,t)$ can be rewritten on the interval $(0,1)$ as follows
	\begin{equation}\label{eq:2.4}
		\left\{
		\begin{aligned}
			d X(x,t)&=\frac{K}{4M^{2}}\Delta X(x,t)dt+2K X(x,t)dt+b(X(x,t))dt+ dW^{\rho}(x,t),\\
			X(x,0)&=u(x),\forall{x\in (0,1)},\\
			X(0,t)&=X(1,t)=0,t\in[0,T],
		\end{aligned}
		\right.
	\end{equation}
	where $u(x):=\hat{u}(2Mx-M)$. In this paper, we shall establish quantitative estimates for the solution to SPDE~\eqref{eq:2.4} and investigate its long time behavior using the log-Sobolev inequality and the Poincar\'e inequality. 
	\subsection{Elementary concepts and estimates}\label{sec:2.2}
	In this section, we first provide some useful properties of functions $\tilde{\rho}(x,y)$ and $\mathcal{K}(x,y)$.
	\begin{lemma}(see Lemma 1.3.1 in \cite{Dalang2024book}).\label{lem:equiva}
		The function $\tilde{\rho}(x,y)$ defined in Equation~\eqref{eq:2.5}~has the following equivalent expression
		\begin{align}\label{eq:kerneleq}
			\tilde{\rho}(x,y)=\sum_{k=1}^{+\infty}2\sin(k\pi y)\sin(k\pi x)e^{-k^{2}\pi^{2}\rho}=\frac{1}{\sqrt{4\pi\rho}}\sum_{k=-\infty}^{+\infty}\left[e^{-\frac{(2k+y-x)^{2}}{4\rho}}-e^{-\frac{(2k+y+x)^{2}}{4\rho}}\right],
		\end{align}
		for $0\leq x,y\leq 1$ and $\rho>0$. Then there exist the following estimates
		\begin{equation}\label{eq:color}
			\begin{split}
				&\vert\tilde{\rho}(x_{1},x_{2})\vert \leq \frac{1}{\sqrt{\pi\rho}}e^{-\frac{(x_{2}-x_{1})^{2}}{4\rho}}+\frac{\pi^{\frac{3}{2}}\rho^{\frac{1}{2}}}{6},x_{1},x_{2}\in(0,1),\\
				&\vert\mathcal{K}(x_{1},x_{2}) \vert\leq
				\frac{1}{\sqrt{2\pi\rho}}e^{-\frac{(x_{2}-x_{1})^{2}}{8\rho}}+\frac{\sqrt{2}\pi^{\frac{3}{2}}\rho^{\frac{1}{2}}}{6},x_{1},x_{2}\in(0,1),\\
			&\lim_{\rho\rightarrow 0}\sqrt{4\pi\rho}\vert\tilde{\rho}(x,x)\vert=1, \lim_{\rho\rightarrow 0}\sqrt{8\pi\rho} \vert\mathcal{K}(x,x) \vert=1,x\in(0,1).
		\end{split}
\end{equation}
	\end{lemma}
	\begin{proof}
		Consider the homogeneous heat equation on [0,1] with Dirichlet boundary conditions
		\begin{equation}\label{eq:heatdiri}
			\begin{split}
			\left\{
			\begin{aligned}
				&\partial_{t}\tilde{u}(x,t)=\Delta \tilde{u}(x,t),\forall{x\in (0,1)},\\
				&\tilde{u}(0,t)=\tilde{u}(1,t)=0,\\
				&\tilde{u}(x,0)=\tilde{\phi}(x),
			\end{aligned}
			\right.
			\end{split}
		\end{equation}
		where $\tilde{\phi}\in L^{2}([0,1])$ is the initial condition. Then the solution to the above equation can be expanded using the eigenfunctions of the Laplacian operator under Dirichlet boundary conditions
		\begin{align*}
			\tilde{u}(x,t)=&\sum_{k=1}^{+\infty}2\int_{0}^{1}\tilde{\phi}(y)\sin(k\pi y)dy \sin(k\pi x)e^{-k^{2}\pi^{2}t}=\int_{0}^{1}G(t;x,y)\tilde{\phi}(y)dy,
		\end{align*}
		where
		\begin{align*}
			G(t;x,y):=2\sum_{k=1}^{+\infty}\sin(k\pi x)\sin(k\pi y)e^{-k^{2}\pi^{2}t},
		\end{align*}
		is the Green's function of Equation~\eqref{eq:heatdiri}.
		According to Lemma 1.3.1 in \cite{Dalang2024book}, the Green’s function $G(t;x,y)$ has the equivalent expression
		\begin{align*}
			G(t;x,y)=&\sum_{k=1}^{+\infty}2\sin(k\pi y)\sin(k\pi x)e^{-k^{2}\pi^{2}t}
			=\frac{1}{\sqrt{4\pi t}}\sum_{k=-\infty}^{+\infty}\left[e^{-\frac{(2k+y-x)^{2}}{4t}}-e^{-\frac{(2k+y+x)^{2}}{4t}}\right].
		\end{align*}
		Hence for $0\leq x,y\leq 1$ and $\rho>0$
		 \begin{align*}
			\tilde{\rho}(x,y)=\sum_{k=1}^{+\infty}2\sin(k\pi y)\sin(k\pi x)e^{-k^{2}\pi^{2}\rho}=\frac{1}{\sqrt{4\pi\rho}}\sum_{k=-\infty}^{+\infty}\left[e^{-\frac{(2k+y-x)^{2}}{4\rho}}-e^{-\frac{(2k+y+x)^{2}}{4\rho}}\right].
		\end{align*}
		Without loss of generality, let $0< x_{1}< x_{2}< 1$. Then
		\begin{align*}
			\vert \tilde{\rho}(x_{1},x_{2})\vert&\leq\frac{1}{\sqrt{4\pi\rho}}\left(e^{-\frac{(x_{2}-x_{1})^{2}}{4\rho}}-e^{-\frac{(x_{1}+x_{2})^{2}}{4\rho}}\right)+\frac{1}{\sqrt{4\pi\rho}}\sum_{n=1}^{+\infty}\left[e^{-\frac{(2n+x_{2}-x_{1})^{2}}{4\rho}}-e^{-\frac{(2n+x_{1}+x_{2})^{2}}{4\rho}}\right]\\
			&+\frac{1}{\sqrt{4\pi\rho}}\sum_{n=-1}^{-\infty}\left[e^{-\frac{(2n+x_{1}+x_{2})^{2}}{4\rho}}-e^{-\frac{(2n+x_{2}-x_{1})^{2}}{4\rho}}\right]\\
			&\leq \frac{1}{\sqrt{4\pi\rho}}e^{-\frac{(x_{2}-x_{1})^{2}}{4\rho}}+\frac{1}{\sqrt{4\pi\rho}}\sum_{n=1}^{+\infty}e^{-\frac{(2n+x_{2}-x_{1})^{2}}{4\rho}}+\frac{1}{\sqrt{4\pi\rho}}\sum_{n=2}^{+\infty}e^{-\frac{(2n-x_{1}-x_{2})^{2}}{4\rho}}\\
			&+\frac{1}{\sqrt{4\pi\rho}}\left[e^{-\frac{(2-x_{1}-x_{2})^{2}}{4\rho}}-e^{-\frac{(2-x_{2}+x_{1})^{2}}{4\rho}}\right].
		\end{align*}
		Because $0<x_{2}-x_{1}\leq 2-x_{1}-x_{2}$ and $x_{1}+x_{2}<2$, so
		\begin{equation}\label{eq:rhoest}
			\begin{split}
				\vert \tilde{\rho}(x_{1},x_{2})\vert
				&\leq \frac{1}{\sqrt{4\pi\rho}}e^{-\frac{(x_{2}-x_{1})^{2}}{4\rho}}+\frac{1}{\sqrt{4\pi\rho}}\sum_{n=1}^{+\infty}e^{-\frac{n^{2}}{\rho}}+\frac{1}{\sqrt{4\pi\rho}}\sum_{n=2}^{+\infty}e^{-\frac{(n-1)^{2}}{\rho}}\\
				&+\frac{1}{\sqrt{4\pi\rho}}e^{-\frac{(x_{2}-x_{1})^{2}}{4\rho}}-\frac{1}{\sqrt{4\pi\rho}}e^{-\frac{1}{\rho}}\\
				&=\frac{1}{\sqrt{\pi\rho}}e^{-\frac{(x_{2}-x_{1})^{2}}{4\rho}}+\frac{1}{\sqrt{\pi\rho}}\sum_{n=1}^{+\infty}e^{-\frac{n^{2}}{\rho}}-\frac{1}{\sqrt{4\pi\rho}}e^{-\frac{1}{\rho}}\\
				&\leq\frac{1}{\sqrt{\pi\rho}}e^{-\frac{(x_{2}-x_{1})^{2}}{4\rho}}+\frac{1}{\sqrt{\pi\rho}}\sum_{n=1}^{+\infty}\frac{\rho}{n^{2}+\rho}\leq \frac{1}{\sqrt{\pi\rho}}e^{-\frac{(x_{2}-x_{1})^{2}}{4\rho}}+\frac{\pi^{\frac{3}{2}}\rho^{\frac{1}{2}}}{6}.
			\end{split}
		\end{equation}
	By the definition in Equation~\eqref{eq:noisecova}~
		\begin{align*}
			\mathcal{K}(x,y)=\int_{0}^{1}\tilde{\rho}(x,z)\tilde{\rho}(z,y)dz,
		\end{align*}
		then
		\begin{equation}\label{eq:cova}
			\begin{split}
				\mathcal{K}(x,y)=&\int_{0}^{1}\sum_{k=1}^{+\infty}2\sin(k\pi z) \sin(k\pi x)e^{-k^{2}\pi^{2}\rho}\sum_{m=1}^{+\infty}2\sin(m\pi z) \sin(m\pi y)e^{-m^{2}\pi^{2}\rho}dz\\
				=&\sum_{k=1}^{+\infty}\sum_{m=1}^{+\infty}4\int_{0}^{1}\sin(k\pi z) \sin(m\pi z)dz\sin(k\pi x) \sin(m\pi y)e^{-k^{2}\pi^{2}\rho}e^{-m^{2}\pi^{2}\rho}\\
				=&\sum_{k=1}^{+\infty}2\sin(k\pi x) \sin(k\pi y)e^{-2k^{2}\pi^{2}\rho}\\
				=&\frac{1}{\sqrt{8\pi\rho}}\sum_{n=-\infty}^{+\infty}\left[e^{-\frac{(2n+y-x)^{2}}{8\rho}}-e^{-\frac{(2n+y+x)^{2}}{8\rho}}\right],
			\end{split}
		\end{equation}
		where the last equality follows from Equation~\eqref{eq:kerneleq}.
		Replacing $\rho$ by $2\rho$ in~\eqref{eq:rhoest}, we obtain 
		\begin{align*}
			|\mathcal{K}(x_{1},x_{2})|\leq \frac{1}{\sqrt{2\pi\rho}}e^{-\frac{(x_{2}-x_{1})^{2}}{8\rho}}+\frac{\sqrt{2}\pi^{\frac{3}{2}}\rho^{\frac{1}{2}}}{6},
		\end{align*}
		and by Equation~\eqref{eq:cova}, we know that
		\begin{align*}
			\lim_{\rho\rightarrow 0}\sqrt{8\pi\rho} \mathcal{K}(x,x)=1,x\in(0,1).
		\end{align*}
		The rest of the proof can be derived directly.
	\end{proof}
	\begin{lemma}\label{lem:lowerbound}
			Given an integer $n\geq 1$, let $\hat{\mathcal{K}}$ be the $n\times n$ matrix with elements $\hat{\mathcal{K}}_{ij}:=\mathcal{K}(x_{i},x_{j})$ for $1\leq i,j\leq n$ and $x_{i},x_{j}\in(0,1)$. Denote by $\hat{c}(\rho)$ the smallest eigenvalue of $\hat{\mathcal{K}}$ such that 
		\begin{align*}
		\hat{\mathcal{K}}\geq \hat{c}(\rho) \text{I}_{d},
		\end{align*}
		where $\text{I}_{d}$ is the identity matrix, and for two symmetric matrices $A$ and $B$, $A\geq B$ means that $B-A$ is non-negative definite. Then for any $i=1,\cdots,n$, we know that
		\begin{equation}\label{eq:lowerbound}
			\begin{split}
				\hat{c}(\rho)&\geq \frac{1}{\sqrt{8\pi\rho}}-\frac{1}{\sqrt{8\pi\rho}}e^{-\frac{x_{i}^{2}}{8\rho}}+ \frac{1}{\sqrt{8\pi\rho}}\sum_{k\neq 0}\left[e^{-\frac{k^{2}}{8\rho}}-e^{-\frac{(k+x_{i})^{2}}{8\rho}}\right]\\
				&- \frac{1}{\sqrt{8\pi\rho}}\sum_{j\neq i}\sum_{k=-\infty}^{+\infty}\left[e^{-\frac{(2k+x_{j}-x_{i})^{2}}{8\rho}}-e^{-\frac{(2k+x_{j}+x_{i})^{2}}{8\rho}}\right].
			\end{split}
		\end{equation}
	\end{lemma}
	\begin{proof}
			For any $1\leq i\leq n$, by the Gersgorin Circle Theorem in linear algebra and Equation~\eqref{eq:cova}~above, we get
		\begin{align*}
			\hat{c}(\rho)\geq\mathcal{K}(x_{i},x_{i})-\sum_{j\neq i}|\mathcal{K}(x_{i},x_{j})|
			&= \frac{1}{\sqrt{8\pi\rho}}-\frac{1}{\sqrt{8\pi\rho}}e^{-\frac{x_{i}^{2}}{8\rho}}+ \frac{1}{\sqrt{8\pi\rho}}\sum_{k\neq 0}\left[e^{-\frac{k^{2}}{8\rho}}-e^{-\frac{(k+x_{i})^{2}}{8\rho}}\right]\\
			&- \frac{1}{\sqrt{8\pi\rho}}\sum_{j\neq i}\sum_{k=-\infty}^{+\infty}\left[e^{-\frac{(2k+x_{j}-x_{i})^{2}}{8\rho}}-e^{-\frac{(2k+x_{j}+x_{i})^{2}}{8\rho}}\right].
		\end{align*}
	\end{proof}	
	In what follows, we shall denote by $L^{2}(\mathbb{S}^{1})$ the space defined as
	\begin{align*}
		L^{2}(\mathbb{S}^{1}):=\{\phi\in L^{2}([0,1])\mid \phi(0)=\phi(1)=0\}.
	\end{align*}
Let	$e_{k}(x)=\sqrt{2}\sin(k\pi x)$ for $k\geq 1$, the family $\{e_{k}\}_{k\geq 1}$ is a complete orthonormal basis of $L^{2}(\mathbb{S}^{1})$. Define the operator $Q_{\rho}^{\frac{1}{2}}$ on $L^{2}(\mathbb{S}^{1})$ 
	\begin{align}\label{eq:ope}
		(Q_{\rho}^{\frac{1}{2}}\phi)(x):=\int_{0}^{1}\tilde{\rho}(x,y)\phi(y)dy, \phi(y)\in L^{2}(\mathbb{S}^{1}).
	\end{align}
By Equation~\eqref{eq:kerneleq}, we know that 	
	\begin{align*}
		(Q_{\rho}^{\frac{1}{2}}e_{k})(x)=\int_{0}^{1}\tilde{\rho}(x,y)e_{k}(y)dy=&\int_{0}^{1}\sum_{r=1}^{+\infty}2\sin(r\pi y)\sin(r\pi x)e^{-r^{2}\pi^{2}\rho}\sqrt{2}\sin(k\pi y)dy\\
		=&e^{-k^{2}\pi^{2}\rho}\sqrt{2}\sin(k\pi x)=e^{-k^{2}\pi^{2}\rho}e_{k}(x).
	\end{align*}
	Therefore, $\{e_{k}\}_{k\geq 1}$ are also the eigenfunctions of the operator $Q_{\rho}^{\frac{1}{2}}$, associated with the eigenvalues $\{\lambda_{k}(\rho)=e^{-k^{2}\pi^{2}\rho}\}_{k\geq1}$. Obviously,
	$Q_{\rho}^{\frac{1}{2}}$ is a Hilbert-Schmidt operator on $L^{2}(\mathbb{S}^{1})$ and
	\begin{align*}
		(Q_{\rho}\phi)(x)=(Q^{\frac{1}{2}}_{\rho}Q^{\frac{1}{2}}_{\rho}\phi)(x)&=\int_{0}^{1}\tilde{\rho}(x,z)\int_{0}^{1}\tilde{\rho}(z,y)\phi(y)dydz\\
		&=\int_{0}^{1}\int_{0}^{1}\tilde{\rho}(x,z)\tilde{\rho}(z,y)dz\phi(y)dy=\int_{0}^{1}\mathcal{K}(x,y)\phi(y)dy, \phi(y)\in L^{2}(\mathbb{S}^{1}).
	\end{align*}
	Denote by $\mathcal{H}_{\rho}$ the Cameron-Martin space
	\begin{align*}
		\mathcal{H}_{\rho}:=\{\phi\in L^{2}(\mathbb{S}^{1})\mid\Vert \phi\Vert^{2}_{\mathcal{H}_{\rho}}:=\Vert Q^{-1/2}_{\rho}  \phi\Vert^{2}_{L^{2}}<+\infty\}.
	\end{align*}
		Let $(\Omega,\mathcal{F}_{t},\mathcal{F},P)$ be a filtered probability space satisfying the usual hypothesis. Considering a sequence of independent real Brownian motion $B^{k}(t)$ defined on $(\Omega,\mathcal{F}_{t},\mathcal{F},P)$, it is well known that $W^{\rho}(x,t)$ can be expressed explicitly as a series 
		\begin{align}\label{eq:1.5}
			W^{\rho}(x,t)= \sum_{k=1}^{+\infty}\lambda_{k}(\rho) e_{k}(x) B^{k}(t)=\sum_{k=1}^{+\infty}e^{-k^{2}\pi^{2}\rho} \sqrt{2}\sin(k\pi x)B^{k}(t).
		\end{align}
		The above expansion can facilitate the calculations of stochastic convolutions in subsequent sections.
		
	Moreover, let $\partial_{x}\tilde{\rho}(x,y)$ denote the partial derivative of $\tilde{\rho}(x,y)$ with respect to $x$. Define
	\begin{align}\label{eq:partialrho}
		(Q_{\partial\rho}^{\frac{1}{2}}\phi)(x):=\int_{0}^{1}\partial_{x}\tilde{\rho}(x,y)\phi(y)dy, \phi(y)\in L^{2}(\mathbb{S}^{1}).
	\end{align}
	Obviously, for any positive integer $k\geq 1$, we have
		\begin{align*}
		(Q_{\partial\rho}^{\frac{1}{2}}e_{k})(x)=\sqrt{2}\int_{0}^{1}\partial_{x}\tilde{\rho}(x,y)\sin(k\pi y)dy=&2\sqrt{2}\sum_{l=1}^{+\infty}e^{-l^{2}\pi^{2}\rho}l\pi \cos(l\pi x)\int_{0}^{1}\sin(l\pi y)\sin(k\pi y)dy\\
		=&\sqrt{2}e^{-k^{2}\pi^{2}\rho}k\pi \cos(k\pi x).
	\end{align*}
	Therefore, the operator $Q_{\partial\rho}^{\frac{1}{2}}: L^{2}(\mathbb{S}^{1})\rightarrow L^{2}([0,1])$ is a Hilbert-Schmidt operator with norm
		\begin{align*}
		\Vert Q_{\partial\rho}^{\frac{1}{2}}\Vert_{2}=:&\sqrt{\sum_{k=1}^{+\infty}\Vert Q_{\partial\rho}^{\frac{1}{2}}e_{k}\Vert^{2}_{L^{2}}}.
	\end{align*}
	The following lemma provides the estimates for operators $Q_{\rho}$ and $Q_{\partial\rho}^{\frac{1}{2}}$.
		\begin{lemma}\label{cor:trest}
				For $\rho\in(0,1)$, we have
		\begin{equation}\label{eq:trespt}
			\begin{split}
				\frac{e^{-2\pi^{2}}}{4\pi^{2}+1}\frac{1}{\sqrt{\rho}}&\leq Tr(Q_{\rho})=\sum_{k=1}^{+\infty} e^{-2k^{2}\pi^{2}\rho}\leq\frac{1}{12\rho},\\
				\Vert Q_{\partial\rho}^{\frac{1}{2}}\Vert^{2}_{2}&=\frac{1}{2}\sum_{k=1}^{+\infty} k^{2}\pi^{2}e^{-2k^{2}\pi^{2}\rho}\leq\frac{1}{48\rho^{2}}.
			\end{split}
		\end{equation}
		\end{lemma}
		\begin{proof}
			According to~\eqref{eq:kerneleq}~and~\eqref{eq:ope}~, we have
			\begin{align*}
		 Tr(Q_{\rho})=\sum_{k=1}^{+\infty}\langle Q_{\rho}e_{k},e_{k}\rangle_{L^{2}}=\sum_{k=1}^{+\infty} e^{-2k^{2}\pi^{2}\rho}\leq\sum_{k=1}^{+\infty}\frac{1}{2k^{2}\pi^{2}\rho}=\frac{1}{12\rho}.
			\end{align*}
		Because $\rho\in(0,1)$, we get
			\begin{align*}
			 Tr(Q_{\rho})=\sum_{k=1}^{+\infty} e^{-2k^{2}\pi^{2}\rho}\geq&\sum_{k=1}^{+\infty}\int_{k}^{k+1}e^{-2x^{2}\pi^{2}\rho}dx=\int_{1}^{+\infty}e^{-2x^{2}\pi^{2}\rho}dx=\frac{1}{2\pi\sqrt{\rho}}\int_{2\pi\sqrt{\rho}}^{+\infty}e^{-\frac{x^{2}}{2}}dx\\
			 \geq&\frac{1}{2\pi\sqrt{\rho}}\int_{2\pi}^{+\infty}e^{-\frac{x^{2}}{2}}dx\geq \frac{1}{\sqrt{\rho}}\frac{e^{-2\pi^{2}}}{4\pi^{2}+1}.
			\end{align*}
			Note that $e^{x}\geq x^{2}$ for $x\in[1,+\infty)$, then
			\begin{align*}
				\Vert Q_{\partial\rho}^{\frac{1}{2}}\Vert^{2}_{2}=&\sum_{k=1}^{+\infty}\langle Q_{\partial\rho}^{\frac{1}{2}}e_{k},Q_{\partial\rho}^{\frac{1}{2}}e_{k}\rangle_{L^{2}}
				=\frac{1}{2}\sum_{k=1}^{+\infty}k^{2}\pi^{2} e^{-2k^{2}\pi^{2}\rho}\leq\sum_{k=1}^{+\infty}\frac{k^{2}\pi^{2}}{8k^{4}\pi^{4}\rho^{2}}=\frac{1}{48\rho^{2}}.
			\end{align*}
		\end{proof}
	
	\section{Regularity for the solution to SPDE}\label{sec:3}
	In this section, we shall analyze the regularity of the solution to SPDE~\eqref{eq:2.4}. The following lemma is a fundamental tool in analysis.
	\begin{lemma}(Gronwall inequality; see Theorem 1 in \cite{Dragomir2003document}).\label{lem:3.2}
		For fixed $T >0$, let $g(t)$ and $h(t)$ be real continuous functions defined on $[0,T]$. If there is a continuous function $l(t)\geq 0$ such that for $t\in[0,T]$
		\begin{align*}
			g(t)\leq h(t)+\int_{0}^{t}l(s)g(s)ds,
		\end{align*}
		then
		\begin{align*}
			g(t)\leq h(t)+\int_{0}^{t}h(\tau)l(\tau)e^{\int_{\tau}^{t}l(s)ds}d\tau.
		\end{align*}
	\end{lemma}
	
	\begin{theorem}\label{thm:3.2}
		Suppose that $X(x,t)$ is the solution to SPDE~\eqref{eq:2.4}, and let $X(t):=X(\cdot,t)$. Then we have the following estimate
		\begin{align*}
			E[\Vert X(t) \Vert^{2}_{L^{2}}]\leq \Vert u\Vert _{L^{2}}^{2}e^{C_{1}t}+\frac{Tr(Q_{\rho})}{C_{1}}(e^{C_{1}t}-1), \ t\in(0,T),
		\end{align*}
		where
		\begin{align*}
			C_{1}:=2(1-\gamma+\frac{1}{\varepsilon}+2K).
		\end{align*}
	\end{theorem}
	\begin{proof}
		By It\^{o} formula, 
		\begin{align*}
			d\Vert X(t)\Vert _{L^{2}}^{2}&=2\langle X(t), dX(t)\rangle_{L^{2}}+Tr(Q_{\rho})dt.
		\end{align*}
		Then
		\begin{align*}
			\Vert X(t)\Vert _{L^{2}}^{2}
			&=\Vert u\Vert _{L^{2}}^{2}+Tr(Q_{\rho})t+2\int_{0}^{t}\langle X(s), b(X(s))\rangle_{L^{2}}ds+4K\int_{0}^{t}\Vert X(s)\Vert _{L^{2}}^{2}ds\\
			&+2\frac{K}{4M^{2}}\int_{0}^{t}\langle X(s), \Delta X(s)\rangle_{L^{2}}ds+2\int_{0}^{t}\langle X(s), dW^{\rho}(s)\rangle_{L^{2}}.
		\end{align*}
		Note that
		\begin{align*}
			2\frac{K}{4M^{2}}\int_{0}^{t}\langle X(s), \Delta X(s)\rangle_{L^{2}}ds=-\frac{K}{2M^{2}}\int_{0}^{t}\Vert \partial_{x} X(s)\Vert _{L^{2}}^{2}ds.
		\end{align*}
		Then
		\begin{align*}
			\Vert X(t)\Vert _{L^{2}}^{2}
			&=\Vert u\Vert _{L^{2}}^{2}+Tr(Q_{\rho})t+2\int_{0}^{t}\langle X(s), b(X(s))\rangle_{L^{2}}ds+4K\int_{0}^{t}\Vert X(s)\Vert _{L^{2}}^{2}ds\\
			&-\frac{K}{2M^{2}}\int_{0}^{t}\Vert \partial_{x} X(s)\Vert _{L^{2}}^{2}ds+2\int_{0}^{t}\langle X(s), dW^{\rho}(s)\rangle_{L^{2}}\\
			&\leq \Vert u\Vert _{L^{2}}^{2}+Tr(Q_{\rho})t+2\int_{0}^{t}\langle X(s), b(X(s))\rangle_{L^{2}}ds+4K\int_{0}^{t}\Vert X(s)\Vert _{L^{2}}^{2}ds\\
			&+2\int_{0}^{t}\langle X(s), dW^{\rho}(s)\rangle_{L^{2}}
			,
		\end{align*}
		and
		\begin{align*}
			E[\Vert X(t)\Vert _{L^{2}}^{2}]
			&\leq \Vert u\Vert _{L^{2}}^{2}+Tr(Q_{\rho})t+2\int_{0}^{t}E[\langle X(s), b(X(s))\rangle_{L^{2}}]ds+4K\int_{0}^{t}E[\Vert X(s)\Vert _{L^{2}}^{2}]ds\\
			&\leq \Vert u\Vert _{L^{2}}^{2}+Tr(Q_{\rho})t+\int_{0}^{t}2(1-\gamma  + 1/\varepsilon+2K)E[\Vert X(s)\Vert _{L^{2}}^{2}]ds\\
			&=\Vert u\Vert _{L^{2}}^{2}+Tr(Q_{\rho})t+\int_{0}^{t}C_{1}E[\Vert X(s)\Vert _{L^{2}}^{2}]ds.
		\end{align*}
		Since the regularized parameter $\varepsilon$ will tend to 0, we can choose $\varepsilon>0$
		to be sufficiently small such that $C_{1}>0$. From the Gronwall's inequality in Lemma \ref{lem:3.2}, we know that
		\begin{align*}
			E[\Vert X(t)\Vert _{L^{2}}^{2}]
			&\leq \Vert u\Vert _{L^{2}}^{2}+Tr(Q_{\rho})t+C_{1}\int_{0}^{t}E[\Vert X(s)\Vert _{L^{2}}^{2}]ds\\
			&\leq \Vert u\Vert _{L^{2}}^{2}e^{C_{1}t}+\frac{Tr(Q_{\rho})}{C_{1}}(e^{C_{1}t}-1), \ t\in(0,T).
		\end{align*}

	\end{proof}
	To establish the estimates for $\partial_{x}X(x,t)$, we first derive the equation it satisfies.
		\begin{theorem}\label{thm:61}
		Suppose that $X(x,t)$ is the solution to SPDE~\eqref{eq:2.4}~. Then for $t\in[0,T]$, the map $x\longmapsto X(x,t)$ is smooth with respect to $x$ almost surely, and the partial derivative  $v(x,t):=\partial_{x}X(x,t)$ satisfies the following SPDE	\begin{align}\label{eq:6.2}
			\left\{
			\begin{aligned}
				dv(x,t)&=\frac{K}{4M^{2}}\Delta  v(x,t)dt+ 2K v(x,t)dt+b^{\prime}(X(x,t))v(x,t)dt+dW^{\partial\rho}(x,t),\\
				v(x,0)&=\partial_{x}u(x),\forall{x\in (0,1)},
			\end{aligned}
			\right.
		\end{align}
		where
		\begin{align*}
			dW^{\partial\rho}(x,t):=\int_{0}^{1} \partial_{x}\tilde{\rho}(x,y)dW(y,t). 
		\end{align*}

	\end{theorem}
	\begin{proof}Firstly, we investigate the regularity of the solution to SPDE~\eqref{eq:2.4}. Let $S(t)$ be the semi-group corresponding to the Laplacian operator $\frac{K}{4M^{2}}\Delta$ with Dirichlet boundary conditions. We consider the stochastic convolution firstly
		\begin{align*}
			W_{L}^{\rho}(x,t):=\int_{0}^{t}S(t-s)dW^{\rho}(x,s)=\sum_{k=1}^{+\infty}\int_{0}^{t} e^{-(t-s)\frac{K\pi^{2}}{4M^{2}}k^{2}}e^{-k^{2}\pi^{2}\rho}\sqrt{2}\sin(k\pi x)dB_{s}^{k}.
		\end{align*}
		For any integer $m\geq 1$,
		\begin{align*}
		E\left[\sum_{k=1}^{+\infty}(1+k^{2})^{m}\vert\int_{0}^{t} e^{-(t-s)\frac{K\pi^{2}}{4M^{2}}k^{2}}e^{-k^{2}\pi^{2}\rho}dB_{s}^{k}\vert^{2}\right]
		\leq&\sum_{k=1}^{+\infty}(1+k^{2})^{m}\int_{0}^{t}e^{-(t-s)\frac{K\pi^{2}}{2M^{2}}k^{2}}e^{-2k^{2}\pi^{2}\rho}ds\\
		\leq&\sum_{k=1}^{+\infty}\frac{2M^{2}(1+k^{2})^{m}}{k^{2}\pi^{2}K}e^{-2k^{2}\pi^{2}\rho}<+\infty,
	    \end{align*}
	    then $W_{L}^{\rho}(\cdot,t)\in H^{m}([0,1])$ almost surely, where $ H^{m}([0,1])$ denotes the Sobolev space consisting of all functions that have weak derivatives up to order $m$. Hence by the Sobolev embedding theorem, the stochastic convolution $W_{L}^{\rho}(\cdot,t)\in C^{\hat{\gamma}}([0,1]),\hat{\gamma}<m-\frac{1}{2}$, where $C^{\hat{\gamma}}([0,1])$ is the space of $\hat{\gamma}$-H\"{o}lder continuous functions. Following an argument analogous to the Theorem 7.8 in \cite{Hairer1998document}, we can conclude that Equation~\eqref{eq:2.4} has a globally unique solution, and the solution is smooth with respect to the spatial variable $x$. Let $C_{c}^{\infty}((0,1))$ denote the set of all infinitely differentiable real-valued functions on $(0,1)$ with compact support. Taking a test function $\phi(x)\in C_{c}^{\infty}((0,1))$, then the solution to Equation~\eqref{eq:2.4}~satisfies
		\begin{align*}
			\langle X(t),\phi\rangle_{L^{2}}&=\langle X(0),\phi\rangle_{L^{2}}+\frac{K}{4M^{2}}\int_{0}^{t} \langle \Delta X(s),\phi\rangle_{L^{2}}ds+2K\int_{0}^{t} \langle X(s),\phi\rangle_{L^{2}}ds\\
			&+\int_{0}^{t}\langle b(X(s)),\phi\rangle_{L^{2}}ds+\int_{0}^{t}\langle \phi,dW^{\rho}(s)\rangle_{L^{2}}\\
			&=\langle X(0),\phi\rangle_{L^{2}}+\frac{K}{4M^{2}}\int_{0}^{t} \langle X(s),\Delta \phi\rangle_{L^{2}}ds+2K\int_{0}^{t} \langle X(s),\phi\rangle_{L^{2}}ds\\
			&+\int_{0}^{t}\langle b(X(s)),\phi\rangle_{L^{2}}ds+\int_{0}^{t}\int_{0}^{1} \langle \phi,\tilde{\rho}(\cdot,y)\rangle_{L^{2}}dW(y,s),\text{a.s.}.
		\end{align*}
		For small $h$, let $X_{h}(t):=X(\cdot+h,t)$. We also have
		\begin{align*}
			\langle X_{h}(t),\phi\rangle_{L^{2}}&=\langle X_{h}(0),\phi\rangle_{L^{2}}+\frac{K}{4M^{2}}\int_{0}^{t} \langle X_{h}(s),\Delta \phi\rangle_{L^{2}}ds+\int_{0}^{t}\langle b(X_{h}(s)),\phi\rangle_{L^{2}}ds+2K\int_{0}^{t} \langle X_{h}(s),\phi\rangle_{L^{2}}ds\\
			&+\int_{0}^{t}\int_{0}^{1} \langle \phi,\tilde{\rho}(\cdot+h,y)\rangle_{L^{2}}dW(y,s),\text{a.s.}.
		\end{align*}
		Let 
		\begin{align*}
			\hat{\nabla}_{h}X(x,t):=\frac{X(x+h,t)-X(x,t)}{h},   
		\end{align*}
		then
		\begin{align*}
			\langle \hat{\nabla}_{h}X(t),\phi\rangle_{L^{2}}&=\langle \hat{\nabla}_{h}X(0),\phi\rangle_{L^{2}}+\frac{K}{4M^{2}}\int_{0}^{t} \langle \hat{\nabla}_{h}X(s),\Delta \phi\rangle_{L^{2}}ds\\
			&+2K\int_{0}^{t} \langle \hat{\nabla}_{h}X(s), \phi\rangle_{L^{2}}ds+\int_{0}^{t}\langle \frac{b(X_{h}(s))-b(X(s))}{h},\phi\rangle_{L^{2}}ds\\
			&+\int_{0}^{t}\int_{0}^{1} \langle \phi,\frac{\tilde{\rho}(\cdot+h,y)-\tilde{\rho}(\cdot,y)}{h}\rangle_{L^{2}}dW(y,s),\text{a.s.}.
		\end{align*}
		By the regularity of the solution and Lebesgue's dominated convergence theorem, we know that
		\begin{align*}
		\lim\limits_{h\rightarrow 0}	\hat{\nabla}_{h}X(x,t)=\lim\limits_{h\rightarrow 0}\frac{X(x+h,t)-X(x,t)}{h}=\partial_{x} X(x,t),\text{a.s.},  
		\end{align*}
		and
		\begin{align*}
			\langle \partial_{x} X(t),\phi\rangle_{L^{2}}&=\langle \partial_{x} X(0),\phi\rangle_{L^{2}}+\frac{K}{4M^{2}}\int_{0}^{t} \langle \Delta\partial_{x} X(s), \phi\rangle_{L^{2}}ds\\
			&+2K\int_{0}^{t} \langle \partial_{x} X(s), \phi\rangle_{L^{2}}ds+\int_{0}^{t}\langle b^{\prime}(X(s))\partial_{x}X(s),\phi\rangle_{L^{2}}ds\\
			&+\int_{0}^{t}\int_{0}^{1} \langle \phi,\partial_{x}\tilde{\rho}(\cdot,y)\rangle_{L^{2}}dW(y,s),\text{a.s.}.
		\end{align*}
		Then 
		\begin{align*}
			d\partial_{x}X(x,t)&=\frac{K}{4M^{2}}\Delta  \partial_{x} X(x,t)dt+ 2K\partial_{x}X(x,t)dt+
			b^{\prime}(X(x,t))\partial_{x}X(x,t)dt+dW^{\partial\rho}(x,t).
		\end{align*}
	\end{proof}	
	Now, we can establish estimates for $\partial_{x}X(x,t)$.
	\begin{theorem}\label{thm:diffestima}
	Suppose that $X(x,t)$ is the solution to SPDE~\eqref{eq:2.4}~and $\partial_{x}X(x,t)$ satisfies Equation~\eqref{eq:6.2}. Then we obtain the following estimates for $t\in[0,T]$
		\begin{align*}
			E[\Vert\partial_{x} X(t)\Vert _{L^{2}}^{2}]
			&\leq \Vert\partial_{x} u\Vert _{L^{2}}^{2}+\Vert Q_{\partial\rho}^{\frac{1}{2}}\Vert^{2}_{2}t,\\
			E[\Vert \partial_{x}X(t)\Vert _{L^{2}}^{4}]
			&\leq\mathcal{R}_{1}(t):=\Vert\partial_{x} u\Vert _{L^{2}}^{4}+\left(8K+6\Vert Q_{\partial\rho}^{\frac{1}{2}}\Vert^{2}_{2}\right)\left(\Vert\partial_{x} u\Vert _{L^{2}}^{2}t+\frac{\Vert Q_{\partial\rho}^{\frac{1}{2}}\Vert^{2}_{2}t^{2}}{2}\right).
		\end{align*}

	\end{theorem}
	
	\begin{proof}
		By It\^{o} formula and Equation~\eqref{eq:6.2}, we have
		\begin{align*}
			d\Vert \partial_{x}X(t)\Vert _{L^{2}}^{2}&=2\langle \partial_{x}X(t), d(\partial_{x} X(t))\rangle_{L^{2}}+\Vert Q_{\partial\rho}^{\frac{1}{2}}\Vert^{2}_{2}dt,
		\end{align*}
		where 
		\begin{align*}
		&E\left[\langle \sum_{k=1}^{+\infty}\int_{0}^{1} \partial_{x}\tilde{\rho}(\cdot,y)\sin(k\pi y)dydB^{k}(t),\sum_{k=1}^{+\infty}\int_{0}^{1} \partial_{x}\tilde{\rho}(\cdot,y)\sin(k\pi y)dydB^{k}(t)\rangle_{L^{2}}^{2}\right]\\
			=&\sum_{k=1}^{+\infty}\langle \int_{0}^{1} \partial_{x}\tilde{\rho}(\cdot,y)\sin(k\pi y)dy,\int_{0}^{1} \partial_{x}\tilde{\rho}(\cdot,y)\sin(k\pi y)dy\rangle_{L^{2}}^{2}dt=\sum_{k=1}^{+\infty}e^{-2k^{2}\pi^{2}\rho}\frac{k^{2}\pi^{2}}{2}dt=\Vert Q_{\partial\rho}^{\frac{1}{2}}\Vert^{2}_{2}dt.
		\end{align*}
		Note that
		\begin{align*}
			\frac{K}{2M^{2}}\int_{0}^{t}\langle  \partial_{x} X(s), \Delta  \partial_{x} X(s)\rangle_{L^{2}}^{2}ds=-\frac{K}{2M^{2}}\int_{0}^{t}\Vert  \partial_{xx} X(s)\Vert _{L^{2}}^{2}ds.
		\end{align*}
		Then
		\begin{align*}
			\Vert \partial_{x} X(t)\Vert _{L^{2}}^{2}
			&=\Vert\partial_{x} u\Vert _{L^{2}}^{2}+\Vert Q_{\partial\rho}^{\frac{1}{2}}\Vert^{2}_{2}t+2\int_{0}^{t}\langle 
			\partial_{x} X(s),b^{\prime}(X(s))\partial_{x} X(s)\rangle_{L^{2}}ds\\
			&+2\frac{K}{4M^{2}}\int_{0}^{t}\langle \partial_{x} X(s), \Delta \partial_{x} X(s) \rangle_{L^{2}}ds+4K\int_{0}^{t}\Vert \partial_{x} X(s)\Vert _{L^{2}}^{2}ds\\
			&+2\int_{0}^{t}\langle \partial_{x} X(s), dW^{\partial\rho}(s)\rangle_{L^{2}}ds\\
			&\leq\Vert \partial_{x} u\Vert _{L^{2}}^{2}+\Vert Q_{\partial\rho}^{\frac{1}{2}}\Vert^{2}_{2}t+2\int_{0}^{t}\langle \partial_{x} X(s),  b^{\prime}(X(s))\partial_{x} X(s)\rangle_{L^{2}}ds\\
			&+4K\int_{0}^{t}\Vert \partial_{x} X(s)\Vert _{L^{2}}^{2}ds+2\int_{0}^{t}\langle\partial_{x} X(s), dW^{\partial\rho}(s)\rangle_{L^{2}}.
		\end{align*}
		Taking the expectation, we get
		\begin{align*}
			E[\Vert\partial_{x} X(t)\Vert _{L^{2}}^{2}]
			&\leq \Vert\partial_{x} u\Vert _{L^{2}}^{2}+\Vert Q_{\partial\rho}^{\frac{1}{2}}\Vert^{2}_{2}t+4K\int_{0}^{t}E[\Vert\partial_{x} X(s)\Vert _{L^{2}}^{2}]ds\\
			&+\int_{0}^{t}2(1-\gamma -\varepsilon)E[\langle \partial_{x} X(s), \partial_{x} X(s)\rangle_{L^{2}}]ds\\
			&\leq \Vert\partial_{x} u\Vert _{L^{2}}^{2}+\Vert Q_{\partial\rho}^{\frac{1}{2}}\Vert^{2}_{2}t+C_{2}\int_{0}^{t}E[\Vert\partial_{x} X(s)\Vert _{L^{2}}^{2}]ds,
		\end{align*}
		where
		\begin{align*}
			C_{2}:=2(1-\gamma -\varepsilon+2K).
		\end{align*}
	Since the coupling parameter $\gamma>1$ is arbitrary, and both the regularized parameter $\varepsilon$ and the coupling parameter $K$ tend to 0, it follows that $C_{2}\leq 0$. Therefore, we know that
		\begin{align*}
			E[\Vert\partial_{x} X(t)\Vert _{L^{2}}^{2}]
			&\leq \Vert\partial_{x} u\Vert _{L^{2}}^{2}+\Vert Q_{\partial\rho}^{\frac{1}{2}}\Vert^{2}_{2}t, \ t\in(0,T).
		\end{align*}
		Furthermore, let $Y_{t}:=\Vert\partial_{x} X(t)\Vert _{L^{2}}^{2}$, then
		\begin{align*}
			dY_{t}^{2}&=2 Y_{t} dY_{t}+d[Y]_{t}=4Y_{t}\langle \partial_{x}X(t), d\partial_{x} X(t)\rangle_{L^{2}}+2\Vert Q_{\partial\rho}^{\frac{1}{2}}\Vert^{2}_{2}Y_{t}dt+d[Y]_{t}\\
			&=4Y_{t}\langle \partial_{x} X(t), b^{\prime}(X(t))\partial_{x} X(t)\rangle_{L^{2}}dt+8K\Vert \partial_{x} X(t)\Vert _{L^{2}}^{2}dt+d[Y]_{t}\\
			&+\frac{K}{M^{2}}Y_{t}\langle \partial_{x} X(t), \Delta \partial_{x} X(t)\rangle_{L^{2}}dt+2\Vert Q_{\partial\rho}^{\frac{1}{2}}\Vert^{2}_{2}Y_{t}dt+4Y_{t}\langle\partial_{x} X(t), dW^{\partial\rho}(t)\rangle_{L^{2}}\\
			&\leq 4\left(1-\gamma -\varepsilon\right)Y_{t}\Vert\partial_{x} X(t)\Vert _{L^{2}}^{2}dt+4Y_{t}\langle \partial_{x} X(t), dW^{\partial\rho}(t)\rangle_{L^{2}}\\
			&+\left(8K+2\Vert Q_{\partial\rho}^{\frac{1}{2}}\Vert^{2}_{2}\right)\Vert \partial_{x} X(t)\Vert _{L^{2}}^{2}dt+d[Y]_{t}, 
		\end{align*}
		where
		\begin{align*}
			[Y]_{t}
			&=4\sum_{k=1}^{+\infty}e^{-2k^{2}\pi^{2}\rho}k^{2}\pi^{2}\int_{0}^{t}\int_{0}^{1}\left(\partial_{x} X(x,s) \cos(k\pi x)\right)^{2}dxds\\
			&\leq 2\sum_{k=1}^{+\infty}k^{2}\pi^{2}e^{-2k^{2}\pi^{2}\rho}\int_{0}^{t}\Vert \partial_{x} X(s) \Vert_{L^{2}}^{2}ds=4\Vert Q_{\partial\rho}^{\frac{1}{2}}\Vert^{2}_{2}\int_{0}^{t}\Vert \partial_{x} X(s) \Vert_{L^{2}}^{2}ds.
		\end{align*}
		Then
		\begin{align*}
			E[\Vert \partial_{x} X(t)\Vert _{L^{2}}^{4}]
			&\leq\Vert\partial_{x} u\Vert _{L^{2}}^{4}+\left(8K+2\Vert Q_{\partial\rho}^{\frac{1}{2}}\Vert^{2}_{2}\right)\int_{0}^{t}E[\Vert \partial_{x} X(s)\Vert _{L^{2}}^{2}]ds\\
			&+E[[Y]_{t}]+4\left(1-\gamma -\varepsilon\right)\int_{0}^{t}E[\Vert\partial_{x} X(s)\Vert _{L^{2}}^{4}]ds\\
			&\leq\Vert\partial_{x} u\Vert _{L^{2}}^{4}+\left(8K+6\Vert Q_{\partial\rho}^{\frac{1}{2}}\Vert^{2}_{2}\right)\int_{0}^{t}E[\Vert \partial_{x} X(s)\Vert _{L^{2}}^{2}]ds\\
			&+4\left(1-\gamma -\varepsilon\right)\int_{0}^{t}E[\Vert\partial_{x} X(s)\Vert _{L^{2}}^{4}]ds\leq \mathcal{R}_{1}(t),
		\end{align*}
		where
		\begin{align*}
			\mathcal{R}_{1}(t)&:=\Vert\partial_{x} u\Vert _{L^{2}}^{4}+\left(8K+6\Vert Q_{\partial\rho}^{\frac{1}{2}}\Vert^{2}_{2}\right)\left(\Vert\partial_{x} u\Vert _{L^{2}}^{2}t+\frac{\Vert Q_{\partial\rho}^{\frac{1}{2}}\Vert^{2}_{2}t^{2}}{2}\right).
		\end{align*}
	This completes the proof.
	\end{proof}
	\section{Clark-Ocone-Haussmann formula and the log-Sobolev inequality}\label{sec:4}
	In this section, we aim to determine the renormalization relation via the log-Sobolev inequality up to the terminal time $T$ for the solution to SPDE~\eqref{eq:2.4}. To this end, we shall derive the Clark-Ocone-Haussmann formula for cylindrical functions on the solution in Section \ref{sec:4.1}. Having the Clark-Ocone-Haussmann formula in hand, we prove the log-Sobolev inequality and the Poincar\'e inequality 
	in Section \ref{sec:4.2}. In particular, the constant in the log-Sobolev inequality depends on the regularized and coupling parameters in SPDE~\eqref{eq:2.4}~as well as the terminal time $T$.  Based on the principle that the constant in the log-Sobolev inequality remains finite as $T\rightarrow+\infty$, the renormalization relation will be determined. Before proceeding to the details, we introduce the following definitions for convenience. Let
	\begin{align*} \mathcal{L}:=\{l:[0,T]\rightarrow L^{2}([0,1])\mid l\text{ is continuous}\},
	\end{align*}
	and
	\begin{align}\label{eq:space}
		L_{a}^{2}:=\{j:\Omega\times[0,T]\rightarrow\mathbb{R}\mid j(\cdot,t)\in\mathcal{F}_{t}\ \text{for all}\ t\in[0,T]\ \text{and} \ E[\int_{0}^{T}|j(t)|^{2}dt]<+\infty\},
	\end{align}
		where $j(\cdot,t)\in\mathcal{F}_{t}$ means that $j(\cdot,t)$ is measurable with respect to $\mathcal{F}_{t}$. For convenience, we will sometimes omit the dot (i.e., 
	$\cdot$) notation.
	 Let $\tilde{\Omega}:=\Omega\times[0,1]\times[0,T]$, we set
	 \begin{align*}
	 	\mathcal{H}_{\rho}\times L^{2}:=\{r\in\mathcal{L}\mid \Vert r\Vert^{2}_{\mathcal{H}_{\rho}\times L^{2}}:= \int_{0}^{T}\Vert r(t)\Vert^{2}_{\mathcal{H}_{\rho}}dt<+\infty\},
	 \end{align*}
	 and
	\begin{align*}
		\tilde{\mathcal{A}}\mathcal{H}_{\rho}:=\{k:\tilde{\Omega}\rightarrow\mathbb{R}\mid k(x,t)\in \mathcal{F}_{t}\ \text{for all} \ x,t\ \text{and}\ k(\omega) \in \mathcal{H}_{\rho}\times L^{2}\ \text{for all} \ \omega\}.
	\end{align*}
		\begin{definition} 
		 A function $F:\mathcal{L}\rightarrow\mathbb{R}$ is said to be a smooth cylindrical function if it is in the form
		\begin{align}\label{eq:cylfun}
			F(l):=f(l(x_{1},t_{1}),\cdots,l(x_{n},t_{n})),
		\end{align}
		where $t_{1},\cdots,t_{n}\in [0,T]$, $x_{1},\cdots,x_{n}\in (0,1)$ and $f$ is a smooth real-valued function on $\mathbb{R}^{n}$. Define
		\begin{align*}
			\mathcal{H}_{\rho}\times \mathbb{H}^{1}:=\{l\in\mathcal{L}\mid \Vert l\Vert^{2}_{\mathcal{H}_{\rho}\times \mathbb{H}^{1}}:= \int_{0}^{T}\Vert\dot{l}(t)\Vert^{2}_{\mathcal{H}_{\rho}}dt<+\infty\},
		\end{align*}
		where the dot is used to denote the derivative with respect to time $t$.
		 For $h\in \mathcal{H}_{\rho}\times \mathbb{H}^{1}$, the directional derivative of $F$ along $h$ is defined as 
		\begin{align*}
			D_{h}F(l):=\lim_{r\rightarrow 0}\frac{F(l+rh)-F(l)}{r}.
		\end{align*}
	\end{definition} 
	We denote by $\text{Cyl}(\mathcal{L})$ the set of all smooth cylindrical functions on $ \mathcal{L}$. For a cylindrical function $F\in \text{Cyl}(\mathcal{L})$ that satisfies
	\begin{align*} F(X)=f(X(x_{1},t),X(x_{2},t),\cdots ,X(x_{n},t)),
	\end{align*}
	where $X(x,t)$ is the solution to SPDE~\eqref{eq:2.4}~and $f$ is a smooth real-valued function on $\mathbb{R}^{n}$. By definition, the directional derivative of $F$ along $h$ is
	\begin{align*}
		D_{h}F(X)=&\lim_{r\rightarrow 0}\frac{F(X+rh)-F(X)}{r}\\
		=&\lim_{r\rightarrow 0}\frac{f(X(x_{1},t)+rh(x_{1},t),\cdots ,X(x_{n},t)+rh(x_{n},t))-f(X(x_{1},t),\cdots ,X(x_{n},t))}{r}\\
		=&\sum_{i=1}^{n}\nabla_{i}fh(x_{i},t),
	\end{align*}
	where 
	\begin{align*}
		\nabla_{i}f:=\nabla_{i}f(X(x_{1},t),X(x_{2},t),\cdots ,X(x_{n},t)).
	\end{align*}
 On the one hand, 
	\begin{align*}
		D_{h}F(X)&=\sum_{i=1}^{n}\int_{0}^{T}\nabla_{i}fI_{[0,t]}(s)\dot{h}(x_{i},s)ds:=\sum_{i=1}^{n}\int_{0}^{T}(D_{i}F)_{s}^{\cdot}\dot{h}(x_{i},s)ds,
	\end{align*}
	where
	\begin{align*}
		D_{i}F(s):=\nabla_{i}f\cdot s\land t,(D_{i}F)_{s}^{\cdot}:=(	D_{i}F(s))^{\cdot}=\nabla_{i}fI_{[0,t]}(s).
	\end{align*}
	On the other hand, by the Riesz representation theorem, we define
	$DF$ as the unique element in $\mathcal{H}_{\rho}\times \mathbb{H}^{1}$ such that
	\begin{align*}
		D_{h}F(X)&=\langle DF,h \rangle_{\mathcal{H}_{\rho}\times \mathbb{H}^{1}}
		=\sum_{i=1}^{n}\int_{0}^{T}(D_{i}F)_{s}^{\cdot}\dot{h}(x_{i},s)ds.
	\end{align*}
	Then we have
	\begin{align*}
		DF(x,s)
		&=\sum_{i=1}^{n}\mathcal{K}(x,x_{i})D_{i}F(s)=\sum_{i=1}^{n}\mathcal{K}(x,x_{i})\nabla_{i}f\cdot s\land t,
	\end{align*}
	and
	\begin{align*}
		\Vert DF\Vert_{\mathcal{H}_{\rho}\times \mathbb{H}^{1}}^{2}
		&=\int_{0}^{T}\int_{0}^{1}Q^{-1/2}_{\rho} ((DF(s,x))^{\cdot})Q^{-1/2}_{\rho} ((DF(s,x))^{\cdot})dxds\\
		&=\int_{0}^{T}\int_{0}^{1}Q^{-1/2}_{\rho} (\sum_{i=1}^{n}\mathcal{K}(x,x_{i})(D_{i}F)_{s}^{\cdot})Q^{-1/2}_{\rho} (\sum_{j=1}^{n}\mathcal{K}(x,x_{j})(D_{j}F)_{s}^{\cdot})dxds\\
		&=\sum_{i=1}^{n}\sum_{j=1}^{n}\int_{0}^{T}(D_{i}F)_{s}^{\cdot}(D_{j}F)_{s}^{\cdot}\int_{0}^{1}Q^{-1/2}_{\rho} (\mathcal{K}(x_{i},x))Q^{-1/2}_{\rho} (\mathcal{K}(x,x_{j}))dxds\\
		&=\sum_{i=1}^{n}\sum_{j=1}^{n}\int_{0}^{T}(D_{i}F)_{s}^{\cdot}(D_{j}F)_{s}^{\cdot}\int_{0}^{1}\rho(x_{i},x)\rho(x,x_{j})dxds\\
		&=\sum_{i=1}^{n}\sum_{j=1}^{n}\mathcal{K}(x_{i},x_{j})\int_{0}^{T}(D_{i}F)_{s}^{\cdot}(D_{j}F)_{s}^{\cdot}ds=t\sum_{i=1}^{n}\sum_{j=1}^{n}\mathcal{K}(x_{i},x_{j})\nabla_{i}f\nabla_{j}f .
	\end{align*}

	\subsection{Integration by parts formula and Clark-Ocone-Haussmann formula}\label{sec:4.1}
	In this section, we will provide an integration by parts formula via the \emph{pull-back formula} of Equation~\eqref{eq:2.4}. For any $r\in(-\hat{\epsilon},\hat{\epsilon})$, let $W^{\rho}(r)(x,t)$ be a perturbation of $W^{\rho}(x,t)$ that satisfies
	\begin{align*}
		W^{\rho}(r)(x,t)=W^{\rho}(x,t)+r\int_{0}^{t}k(x,s)ds, k\in\tilde{\mathcal{A}}\mathcal{H}_{\rho},
	\end{align*}
	and let $\{X^{r}(x,t)\}_{(x,t)\in [0,1]\times [0,T]}$ be the perturbation of solution $\{X(x,t)\}_{(x,t)\in [0,1]\times [0,T]}$ that satisfies
	\begin{align}\label{eq:pur}
		d X^{r}(x,t)=\frac{K}{4M^{2}}\Delta X^{r}(x,t)+2K X^{r}(x,t)+b(X^{r}(x,t))dt+ dW^{\rho}(r)(x,t).
	\end{align}
	\begin{definition}\label{def:tanpro} 
		Let $ \mathcal{PH}$ be the function space that satisfies 
	\begin{align*}
		\mathcal{PH}:=\{h:\tilde{\Omega}\rightarrow\mathbb{R}\mid\dot{h}(x)\in	L_{a}^{2} \ \text{for all}\ x\ \text{and}\ \partial_{xx}h(\omega,x,t)=C h(\omega,x,t)\ \text{for all}\ \omega,x,t\}.
	\end{align*}
	Note that the functions in space $\mathcal{PH}$ are the eigenfunctions of Laplacian operator  with respect to variable $x$ under Dirichlet boundary conditions. Specifically, they can take the following forms: for any positive integer $m$, $\tilde{h}_{m}(x,t)=\sin{(m\pi x)}h_{m}(t)$, where $h_{m}(t)$ belongs to the adapted Cameron-Martin space of the 1D Wiener process, in fact, $C=C(m)=(m\pi)^{2}$.
\end{definition} 
	\begin{lemma}(Pull-back formula).\label{lem:23}
		If the solution to Equation~\eqref{eq:pur}~satisfies 
		\begin{enumerate}[{\rm (i)}]
			\item $(X^{r}(x,t))_{t\in[0,T]}\in \mathcal{L}$ for any $r$, and $X^{r}(x,t)|_{r=0}=X(x,t)$,
			\item $\frac{d}{dr}X^{r}(x,t)|_{r=0}$ exists, and $\frac{d}{dr}X^{r}(x,t)|_{r=0}=h(x,t)$ for $h\in\mathcal{PH}$.
		\end{enumerate}
		Then
		\begin{align*}
			\dot{h}(x,t)=\frac{K}{4M^{2}}\Delta h(x,t)+2K h(x,t) +b^{\prime}(X(x,t))h(x,t)+
			k(x,t),
		\end{align*}
		where
		\begin{align*}
		b^{\prime}(X(x,t))= U^{\prime}_{\varepsilon,\delta}(X(x,t))+(1-\gamma).
 		\end{align*}
	\end{lemma}
	\begin{proof}
		By differentiating Equation~\eqref{eq:pur}~with respect to $r$ at $r = 0$, we get
		\begin{align*}
			d\frac{dX^{r}(x,t)}{dr}|_{r=0}&=\frac{K}{4M^{2}}\Delta \frac{dX^{r}(x,t)}{dr}|_{r=0}dt+2K \frac{dX^{r}(x,t)}{dr}|_{r=0}dt\\
			&+b^{\prime}(X(x,t))\frac{dX^{r}(x,t)}{dr}|_{r=0}dt	
			+d\frac{dW^{\rho}(r)(x,t)}{dr}|_{r=0},
		\end{align*}
		which means
		\begin{align*}
			\dot{h}(x,t)dt=\frac{K}{4M^{2}}\Delta h(x,t)dt+2K h(x,t) +b^{\prime}(X(x,t))h(x,t)dt+
			k(x,t)dt.
		\end{align*}
		The proof is completed. 
	\end{proof}
The following integration by parts formula is more or less known, see Theorem 5.2.1 in \cite{Zhu2022document} for example. For the reader's convenience, we will provide a sketch of the proof here. 
	\begin{theorem}\label{thm:4.3}
		Let $F\in\text{Cyl}(\mathcal{L})$ be a smooth cylindrical function that satisfies
		\begin{align*}
			F(X)=f(X(x_{1},T), X(x_{2},T), \cdots ,X(x_{n},T)),
		\end{align*}
		where $X(x,t)$ is the solution to SPDE~\eqref{eq:2.4}~and $f$ is a smooth real-valued function on $\mathbb{R}^{n}$. Then the integration by parts formula can be obtained
		\begin{align*}
			E[D_{h}F(X)]=E[F(X)\int_{0}^{T}\langle k(\cdot,t),dW^{\rho}(t)\rangle_{\mathcal{H}_{\rho}}],
		\end{align*}
		where $h\in\mathcal{PH}$ and $k(x,t)$ is determined by
		\begin{align*}
			\dot{h}(x,t)&=\frac{K}{4M^{2}}\Delta h(x,t)+2K h(x,t)+b^{\prime}(X(x,t))h(x,t)+k(x,t).
		\end{align*}
	\end{theorem}
	\begin{proof}
		By Definition \ref{def:tanpro} and Lemma \ref{lem:23}, we know that for any $h\in\mathcal{PH}$, the following equation holds
		\begin{align*}
			k(x,t)=\dot{h}(x,t)-\frac{K}{4M^{2}}\Delta h(x,t)-2K h(x,t) -b^{\prime}(X(x,t))h(x,t).
		\end{align*}
		For any positive constant  $\tilde{C}\in\mathbb{R}$, we define the $\mathcal{F}_{t}$-stopping time
	\begin{align*}
		\tilde{\tau}_{\tilde{C}}:=\inf\{t\in[0,T]\mid\Vert k(\cdot,t)\Vert_{\mathcal{H}_{\rho}}^{2}> \tilde{C} \}.
	\end{align*}
	Denote 
	\begin{align*}
		\tilde{Z}_{t}(r)=:\exp\left[-r\int_{0}^{t}\langle k(\cdot,s\wedge\tilde{\tau}_{\tilde{C}}),dW^{\rho}(s)\rangle_{\mathcal{H}_{\rho}}-\frac{r^{2}}{2}\int_{0}^{t}\Vert k(\cdot,s\wedge\tilde{\tau}_{\tilde{C}})\Vert^{2}_{\mathcal{H}_{\rho}}ds\right].
	\end{align*}
	By Novikov’s criterion, we know that for each $r\in(-\hat{\epsilon},\hat{\epsilon})$, the process $\tilde{Z}_{t}(r)$ satisfies $E[\tilde{Z}_{T}(r)]=1$. According to Theorem 10.14 in \cite{Da2014book}, we know that under the probability $d\tilde{Q}=\tilde{Z}_{T}(r)dP$,
		 the process
		 \begin{align*}
		 	\tilde{W}^{\rho}(r)(x,t):=W^{\rho}(x,t)+r\int_{0}^{t}k(x,s\wedge\tilde{\tau}_{\tilde{C}})ds, t\in[0,T],
		 \end{align*}
		 is a Wiener process, with $Q_{\rho}$ as the covariance operator. Let $\tilde{X}^{r}(x,t)$ be the solution to~\eqref{eq:pur}~with $\tilde{W}^{\rho}(r)(x,t)$. According to the uniqueness of distribution for SPDE~\eqref{eq:2.4}, we have $E[F(X)]=E[\tilde{Z}_{T}(r)F(\tilde{X}^{r})]$. Taking the derivative with respect to $r$ at $r=0$, we get
		  \begin{align*}
		  0=&E[\frac{d}{dr}\tilde{Z}_{T}(r)|_{r=0}F(X)]+E[\frac{d}{dr}F(\tilde{X}^{r})|_{r=0}]
		  =-E[F(X)\int_{0}^{T}\langle k(\cdot,s\wedge\tilde{\tau}_{\tilde{C}}),dW^{\rho}(s)\rangle_{\mathcal{H}_{\rho}}]+E[D_{\hat{h}}F],
		  \end{align*}
		  where $\hat{h}(x,t)$ satisfies
		  \begin{align*}
		  	\dot{\hat{h}}(x,t)=\frac{K}{4M^{2}}\Delta \hat{h}(x,t)+2K \hat{h}(x,t) +b^{\prime}(X(x,t))\hat{h}(x,t)+k(x,t\wedge\tilde{\tau}_{\tilde{C}}).
		  \end{align*}
		  Consequently, by taking the limit as $\tilde{C}\rightarrow+\infty$, we obtain the integration by parts formula for any $h\in\mathcal{PH}$ 
		  \begin{align*}
		 	E[D_{h}F]=E[F(X)\int_{0}^{T}\langle k(\cdot,s),dW^{\rho}(s)\rangle_{\mathcal{H}_{\rho}}].
		 \end{align*}
	\end{proof}

Now, we are going to establish the Clark-Ocone-Haussmann formula for cylindrical functions. To prepare for this, we first introduce the following definitions related to the operators to be used in deriving the Clark-Ocone-Haussmann formula. Consider the function space
\begin{align*}
	L_{a,n}^{2}:=
	\{j:\Omega\times[0,T]\rightarrow\mathbb{R}^{n}\mid j(t)\in\mathcal{F}_{t}\ \text{for all}\ t\in[0,T]\ \text{and} \ E[\int_{0}^{T}\Vert j(t)\Vert_{\mathbb{R}^{n}}^{2}dt]<+\infty\}.
\end{align*}	
Define the operator $\mathcal{A}$ on $L_{a,n}^{2}$
\begin{align}\label{eq:operator}
	(\mathcal{A} r)(t):=r(t)-A(t)\int_{0}^{t}r(s)ds, r\in L_{a,n}^{2},
\end{align}
where $A(t)$ is an $n\times n$ symmetric matrix with elements 
\begin{align}\label{eq:matrix}
	A_{ij}(t):=\langle  b^{\prime}\left(X\left( t\right)\right)e_{l_{j}},e_{l_{i}}\rangle_{L^{2}}-\frac{(l_{i}\pi)^{2}}{4M^{2}}K\tilde{\delta}_{ij}+2K\tilde{\delta}_{ij}, 1\leq i,j\leq n.
\end{align}
Here, $l_{i},l_{j}$ are positive integers, and $\tilde{\delta}_{ij}$ denotes the Kronecker function, i.e. $\tilde{\delta}_{ij}=1$ if $i=j$ and 0 otherwise. For $s\in[0,T]$, let $\hat{M}(s)$ denote the $n\times n$ matrix satisfying the following equation
	\begin{align}\label{eq:elemsol}
	\frac{d}{ds}\hat{M}(s)=A(s)\hat{M}(s),
\end{align}
with the initial condition $\hat{M}(0)=\text{I}_{d}$, where $\text{I}_{d}$ is the identity matrix. Let
\begin{align*}
	l(t):=(\mathcal{A} r)(t)=r(t)-A(t)\int_{0}^{t}r(s)ds,
\end{align*}
then $r(t)=l(t)+A(t)\int_{0}^{t}r(s)ds$. Since $\hat{M}(s)$ satisfies Equation~\eqref{eq:elemsol}, we have \begin{align*}
	\int_{0}^{t}r(s)ds=\hat{M}(t)\int_{0}^{t}\hat{M}^{-1}(s)l(s)ds,
\end{align*}
and
\begin{align}\label{eq:opeinvers}
	(\mathcal{A}^{-1}l)(t)=r(t)=l(t)+A(t)\hat{M}(t)\int_{0}^{t}\hat{M}^{-1}(s)l(s)ds,
\end{align}
which means the operator $\mathcal{A}$ is invertible. 
 Now, we shall determine the explicit expression of the operator $	(\mathcal{A}^{\star})^{-1}$, where $\mathcal{A}^{\star}$ is the dual operator of $\mathcal{A}$ in $L_{a,n}^{2}$. Note that from Equations~\eqref{eq:operator},~\eqref{eq:matrix},~\eqref{eq:elemsol},~\eqref{eq:opeinvers}, the matrices $A(t),\hat{M}(t)$ and the operators $\mathcal{A},\mathcal{A}^{\star},\mathcal{A}^{-1}$, introduced as above, depend on the choice of $l_{1},\cdots,l_{n}$.
	\begin{lemma}\label{lem:opera}
		 For the operator $\mathcal{A}$ defined in~\eqref{eq:operator}, the operator $	(\mathcal{A}^{\star})^{-1}$ exists and for any $l\in L_{a,n}^{2}$	
		\begin{align}\label{eq:inveroperator}
			((\mathcal{A}^{\star})^{-1}l)(s)=&l(s)+E\left[(\hat{M}^{*})^{-1}(s)\int_{s}^{T}\hat{M}^{*}(\tau)A(\tau)l(\tau)d\tau|\mathcal{F}_{s}\right],
		\end{align}
		where $\hat{M}(t)$ satisfies Equation~\eqref{eq:elemsol}~and $\hat{M}^{*}(t)$ denotes the transpose of the matrix $\hat{M}(t)$. For any positive integer $N$, we define the $\mathcal{F}_{t}$-stopping time
		\begin{align*}
			\tau_{N}:=\inf\{t\in[0,T]\mid\Vert \partial_{x}X(t)\Vert_{L^{2}}> N \}.
		\end{align*}
		For $s\leq T\wedge\tau_{N}$, let 
		$$J(s):=\hat{M}(T\wedge\tau_{N})\hat{M}^{-1}(s),J(T\wedge\tau_{N})=\text{I}_{d},$$
		where $\text{I}_{d}$ is the identity matrix. Then for $s\leq T\wedge\tau_{N}$, there are  $l_{1}\leq l_{2}\cdots\leq l_{n}$ such that for any $y\in\mathbb{R}^{n}$
		\begin{align}\label{eq:opestimat}
			\Vert J^{*}(s)y\Vert^{2}_{\mathbb{R}^{n}}\leq&e^{(1-\gamma-\varepsilon+4K)(T\wedge\tau_{N}-s)}\Vert y\Vert^{2}_{\mathbb{R}^{n}}.
		\end{align}
	\end{lemma}
	\begin{proof}
		Firstly, we calculate the operator $(\mathcal{A}^{\star})^{-1}$. Note that $(\mathcal{A}^{\star})^{-1}=(\mathcal{A}^{-1})^{\star}$. By Equation~\eqref{eq:opeinvers}, we know that for any $r,l\in L_{a,n}^{2}$,
		\begin{align*}
			E[\int_{0}^{T}\langle((\mathcal{A}^{\star})^{-1}r)(s),l(s)\rangle_{\mathbb{R}^{n}} ds]&=E[\int_{0}^{T}\langle((\mathcal{A}^{-1})^\star r)(s),l(s)\rangle_{\mathbb{R}^{n}} ds]=E[\int_{0}^{T}\langle r(s),(\mathcal{A}^{-1}l)(s)\rangle_{\mathbb{R}^{n}} ds]\\
			&=E\left[\int_{0}^{T}\langle r(s),\left(l(s)+A(s)\hat{M}(s)\int_{0}^{s}\hat{M}^{-1}(\tau)l(\tau)d\tau\right)\rangle_{\mathbb{R}^{n}}ds\right] \\
			&=E\left[\int_{0}^{T}\langle r(s),l(s)\rangle_{\mathbb{R}^{n}}ds\right]\\
			&+E\left[\int_{0}^{T}\langle (\hat{M}^{*})^{-1}(s)\int_{s}^{T}\hat{M}^{*}(\tau)A(\tau)r(\tau)d\tau, l(s)\rangle_{\mathbb{R}^{n}}ds\right],
		\end{align*}
		which implies
		\begin{align*}
			((\mathcal{A}^{\star})^{-1}r)(s)=&r(s)+E\left[(\hat{M}^{*})^{-1}(s)\int_{s}^{T}\hat{M}^{*}(\tau)A(\tau)r(\tau)d\tau|\mathcal{F}_{s}\right].
		\end{align*}
		Moreover, by the definition of $A(s)$, for $1\leq i,j\leq n$
	\begin{align*}
	A_{ij}(s)=\langle  b^{\prime}\left(X\left( s\right)\right)e_{l_{j}},e_{l_{i}}\rangle_{L^{2}}-\frac{(l_{i}\pi)^{2}}{4M^{2}}K\tilde{\delta}_{ij}+2K\tilde{\delta}_{ij},
\end{align*}
and
		\begin{align*}
			\langle  b^{\prime}\left(X\left( s\right)\right)e_{l_{j}},e_{l_{i}}\rangle_{L^{2}}=&2\int_{0}^{1} b^{\prime}\left(X\left(x, s\right)\right)\sin(l_{j}\pi x)\sin(l_{i}\pi x)dx\\
			=&\int_{0}^{1} b^{\prime}\left(X\left(x, s\right)\right)\cos((l_{i}-l_{j})\pi x)dx-\int_{0}^{1} b^{\prime}\left(X\left(x, s\right)\right)\cos((l_{i}+l_{j})\pi x)dx.
		\end{align*}
		 We know that for $i=j$	
		\begin{align*}
			A_{ii}(s)=&\int_{0}^{1} b^{\prime}\left(X\left(x, s\right)\right)dx-\int_{0}^{1} b^{\prime}\left(X\left(x, s\right)\right)\cos(2l_{i}\pi x)dx-\frac{(l_{i}\pi)^{2}}{4M^{2}}K+2K\\
			\leq& 1-\gamma-\varepsilon-\frac{(l_{i}\pi)^{2}}{4M^{2}}K+2K-\int_{0}^{1} b^{\prime}\left(X\left(x, s\right)\right)\cos(2l_{i}\pi x)dx.
		\end{align*}	
		However, for $k> 0$	
		\begin{align*}
			\int_{0}^{1} b^{\prime}\left(X\left(x, s\right)\right)\cos(k\pi x)dx=& \frac{1}{k\pi}	\int_{0}^{1} b^{\prime}\left(X\left(x, s\right)\right)d\sin(k\pi x)\\
			=&\frac{1}{k\pi}b^{\prime}\left(X\left(x, s\right)\right)\sin(k\pi x)|_{x=0}^{x=1}-\frac{1}{k\pi}\int_{0}^{1} b^{\prime\prime}\left(X\left(x, s\right)\right)\partial_{x}X\left(x,s\right)\sin(k\pi x)dx\\
			=&-\frac{1}{k\pi}\int_{0}^{1} b^{\prime\prime}\left(X\left(x, s\right)\right)\partial_{x}X\left(x,s\right)\sin(k\pi x)dx.
		\end{align*}
		Then
		\begin{align*}
			\left|\int_{0}^{1} b^{\prime}\left(X\left(x, s\right)\right)\cos(k\pi x)dx\right|=& \frac{1}{k\pi}	\left|\int_{0}^{1} b^{\prime\prime}\left(X\left(x, s\right)\right)\partial_{x}X\left(x,s\right)\sin(k\pi x)dx\right|\\
			=& \frac{\hat{J}}{k\pi}\frac{1-\varepsilon^{2}}{\varepsilon\delta}	\left|\int_{0}^{1}\textbf{I}_{o_{\delta}(1)\cup o_{\delta}(-1)}(X(x,s)) e^{-\frac{1}{1-\left(\frac{X(x,s)-1}{\delta}\right)^{2}}}\partial_{x}X\left(x,s\right)\sin(k\pi x)dx\right|\\
			\leq&\frac{\hat{J}}{k\pi}\frac{1-\varepsilon^{2}}{\varepsilon\delta}\Vert \partial_{x}X(s)\Vert_{L^{2}},
		\end{align*}
		where $o_{\delta}(x_{0}):=\{x\in\mathbb{R}| -\delta<x-x_{0}<\delta\}$ for $x_{0}\in\mathbb{R}$ and $\textbf{I}_{A}$ denotes the indicator function for $A\subset\mathbb{R}$
		\begin{equation*}
			\textbf{I}_{A}(x):=\begin{cases}
				1 &x\in A,\\
				0& x\notin A.
			\end{cases}
		\end{equation*}
		Recall that for any positive integer $N$,
		\begin{align*}
			\tau_{N}=\inf\{t\in[0,T]\mid\Vert \partial_{x}X(t)\Vert_{L^{2}}> N \}.
		\end{align*}
		Fixing the parameters $\varepsilon,\delta,\gamma,\rho$ and $N$, we can choose positive integers $l_{1}<\cdots<l_{n}$ such that for any $i<j$,
		\begin{align*}
			|l_{i}-l_{j}|\geq \frac{8N(n+1)\hat{J}}{|1-\gamma-\varepsilon|\varepsilon\delta}.
		\end{align*}
		Then for $s\in [0,T\wedge\tau_{N}]$, we have 
		\begin{align*}
			|A_{ij}(s)|\leq \frac{|1-\gamma-\varepsilon|}{4\pi(n+1)},\quad i\neq j.
		\end{align*}
		For $i=j$
		\begin{align*}
			A_{ii}(s)\leq 1-\gamma-\varepsilon-\frac{(l_{i}\pi)^{2}}{4M^{2}}K+2K+ \frac{|1-\gamma-\varepsilon|\varepsilon\delta}{4N(n+1)}\frac{1}{\pi}\frac{1-\varepsilon^{2}}{\varepsilon\delta}N.
		\end{align*}
		Thus, we know that for $s\in [0,T\wedge\tau_{N}]$
		\begin{align*}
			A(s)\leq \left( 1-\gamma-\varepsilon+2K+(2n-1) \frac{|1-\gamma-\varepsilon|}{4\pi(n+1)}\right)\text{I}_{d}\leq\left( \frac{1-\gamma-\varepsilon}{2}+2K\right)\text{I}_{d},
		\end{align*}
		in the order of the non-negative definiteness. Because for any $s\leq T\wedge\tau_{N}$, 
		$$J(s)=\hat{M}(T\wedge\tau_{N})\hat{M}^{-1}(s),J(T\wedge\tau_{N})=\text{I}_{d},$$
		and
		\begin{align*}
		0=\frac{d}{ds}(\hat{M}(s)\hat{M}^{-1}(s))=&(\frac{d}{ds}\hat{M}(s))\hat{M}^{-1}(s)+\hat{M}(s)\frac{d}{ds}(\hat{M}^{-1}(s))\\
		=&A(s)+\hat{M}(s)\frac{d}{ds}(\hat{M}^{-1}(s)),
		\end{align*}
		we have
		\begin{align*}
			\frac{d}{ds}J(s)=&\hat{M}(T\wedge\tau_{N})\frac{d}{ds}(\hat{M}^{-1}(s))=-\hat{M}(T\wedge\tau_{N})\hat{M}^{-1}(s)A(s)=-J(s)A(s),
		\end{align*}
		and
			\begin{align*}
			\frac{d}{ds}J^{*}(s)=&-A(s)J^{*}(s).
		\end{align*}
		Therefore, for any $y\in\mathbb{R}^{n}$
		\begin{equation}\label{eq:stopoperestimete}
			\begin{split}
				&\frac{d}{ds}(e^{(1-\gamma-\varepsilon+4K)s}\Vert J^{*}(s)y\Vert^{2}_{\mathbb{R}^{n}})\\
				=&(1-\gamma-\varepsilon+4K)e^{-(1-\gamma-\varepsilon+4K)s}\Vert J^{*}(s)y\Vert^{2}_{\mathbb{R}^{n}}+2e^{-(1-\gamma-\varepsilon+4K)s}\langle \frac{d}{ds}J^{*}(s)y,J^{*}(s)y\rangle_{\mathbb{R}^{n}}\\
				=&(1-\gamma-\varepsilon+4K)e^{-(1-\gamma-\varepsilon+4K)s}\Vert J^{*}(s)y\Vert^{2}_{\mathbb{R}^{n}}-2e^{-(1-\gamma-\varepsilon+4K)s}\langle A(s)J^{*}(s)y,J^{*}(s)y\rangle_{\mathbb{R}^{n}}\\
				=&2e^{-(1-\gamma-\varepsilon+4K)s}\langle (\frac{1-\gamma-\varepsilon+4K}{2}\text{I}_{d}-A(s))J^{*}(s)y,J^{*}(s)y\rangle_{\mathbb{R}^{n}}\geq 0.
			\end{split}
		\end{equation}	
		Finally, we get
		\begin{align*}
			\Vert J^{*}(s)y\Vert^{2}_{\mathbb{R}^{n}}\leq&e^{(1-\gamma-\varepsilon+4K)(T\wedge\tau_{N}-s)}\Vert y\Vert^{2}_{\mathbb{R}^{n}}.
		\end{align*}
	\end{proof}
Building on the previous preparation, we are now in a position to derive the Clark-Ocone-Haussmann formula. 
	\begin{theorem}\label{thm:finiteresp}
		Suppose that $X(x,t)$ denotes the solution to SPDE~\eqref{eq:2.4}, and let $F\in\text{Cyl}(\mathcal{L})$ be a cylindrical function that satisfies
		\begin{align}\label{eq:cyl}
			F(X)=f(X(x_{1},T), X(x_{2},T), \cdots ,X(x_{n},T)),
		\end{align}
		where $f$ is a smooth real-valued function on $\mathbb{R}^{n}$. Then $F(X)$ can be represented as
		\begin{align}\label{eq:3}
			F(X)=E[F]+\sum_{i=1}^{n}\int_{0}^{T}  H^{F}(x_{l_{i}},s)dB^{l_{i}}(s),
		\end{align}
		where $H^{F}(x_{l_{i}},s)$ denotes the $i$-th component of  $\vec{H}^{F}(s)=(H^{F}(x_{l_{1}},s),\cdots,H^{F}(x_{l_{n}},s))^{T}$. Here, $\vec{H}^{F}(s)$ is an $\mathcal{F}_{t}$-predictable process determined by
				\begin{align*}
				\vec{H}^{F}(s)
				=&\mathbf{Q}^{-1}(\mathcal{A}^{\star})^{-1} E[\vec{D}^{F}(s)|\mathcal{F}_{s}],
			\end{align*}
		and $\mathbf{Q}^{-1},\vec{D}^{F}(s)$ will be defined in~\eqref{eq:vecdefq}.
	\end{theorem}
	\begin{proof}
		By martingale representation theorem for Brownian motion, there exists a predictable process $\vec{H}^{F}(t)$ such that Equation~\eqref{eq:3}~holds. Below we calculate the explicit expression of $\vec{H}^{F}(t)$. On the one hand, by Lemma \ref{lem:23} for any $\tilde{h}_{m}\in \mathcal{PH}$, we know that $\tilde{h}_{m}(x,t)=\sin(m\pi x)h_{m}(t)$ and 
		\begin{align*}
			\sin(m\pi x)\dot{h}_{m}(t)&=\dot{\tilde{h}}_m\left(x, t\right)\\
			&=\frac{K}{4 M^2} \partial_{xx} \tilde{h}_m\left(x, t\right)+2 K \tilde{h}_m\left(x, t\right)+ b^{\prime}\left(X\left(x, t\right)\right) \tilde{h}_m\left(x, t\right)+ k_{m}\left(x, t\right)\\
			&=\sin(m\pi x)\left(-\frac{Km^{2}\pi^{2}}{4M^2}h_{m}(t)+2Kh_{m}\left( t\right)+ b^{\prime}\left(X\left(x, t\right)\right) h_{m}\left( t\right)\right)+ k_{m}\left(x, t\right).
		\end{align*}
		Then by Theorem \ref{thm:4.3}, we know that
		\begin{equation}\label{eq:martin}
			\begin{split}
			E\left[D_{\tilde{h}_{m}} F\right]&=E[F(X)\int_{0}^{T}\langle k_{m}(\cdot,s),dW^{\rho}(s)\rangle_{\mathcal{H}_{\rho}}]\\
			&=E\left[F(X) \int_0^T\dot{h}_{m}(t)\langle\frac{e_{m}}{\sqrt{2}},dW^{\rho}(t)\rangle_{\mathcal{H}_{\rho}}\right]\\
			&-E\left[F(X) \int_0^Th_{m}(t)\langle  b^{\prime}\left(X\left(t\right)\right)\frac{e_{m}}{\sqrt{2}},dW^{\rho}(t)\rangle_{\mathcal{H}_{\rho}}\right]\\
			& -K  E\left[F(X)\int_0^T \left(-\frac{m^{2}\pi^{2}}{4 M^2}+2\right)h_{m}(t)\langle\frac{e_{m}}{\sqrt{2}} ,dW^{\rho}(t)\rangle_{\mathcal{H}_{\rho}}\right].
			\end{split}
		\end{equation}
	Recall that $Q_{\rho}^{-1}\sin(m\pi x)=e^{2m^{2}\pi^{2}\rho}\sin(m\pi x)$, we have
\begin{align*}
	\frac{1}{\sqrt{2}}\langle e_{m},dW^{\rho}(t)\rangle_{\mathcal{H}_{\rho}}=\frac{1}{\sqrt{2}}\langle Q_{\rho}^{-1}e_{m},dW^{\rho}(t)\rangle_{L^{2}}=\frac{1}{\sqrt{2}}\langle e^{2m^{2}\pi^{2}\rho}e_{m},dW^{\rho}(t)\rangle_{L^{2}}=\frac{1}{\sqrt{2}}e^{m^{2}\pi^{2}\rho}dB^{m}(t),
\end{align*}
and
\begin{align*}
	&\langle Q_{\rho}^{-1}\left(b^{\prime}\left(X\left( t\right)\right)e_{m}\right),e_{k}\rangle_{L^{2}}
	=e^{2k^{2}\pi^{2}\rho}\langle b^{\prime}\left(X\left( t\right)\right)e_{m},e_{k}\rangle_{L^{2}}.
\end{align*}
	Let $m$ take values in $\{l_{1},\cdots,l_{n}\}$, as selected in Lemma \ref{lem:opera}. We define the following vectors 
		\begin{equation}\label{eq:vecdefq}
			\begin{split}
			\vec{\dot{h}}(s):=&(\dot{h}_{l_{1}}(s),\cdots,\dot{h}_{l_{n}}(s))^{T}, \\
			\vec{H}^{F}(s):=&(H^{F}(x_{l_{1}},s),\cdots,H^{F}(x_{l_{n}},s))^{T}:=(H_{1}^{F}(s),\cdots,H_{n}^{F}(s))^{T}, \\
			\vec{\hat{H}}^{F}(s):=&(e^{l_{1}^{2}\pi^{2}\rho}H_{1}^{F}(s),\cdots,e^{l_{n}^{2}\pi^{2}\rho}H_{n}^{F}(s))^{T}:=\mathbf{Q}\vec{H}^{F}(s), \\
	    	\vec{D}^{F}(s):=&\sqrt{2}(\sum_{i=1}^{n}\sin(l_{1}\pi x_{i})\nabla_{i} f,\cdots,\sum_{i=1}^{n}\sin(l_{n}\pi x_{i})\nabla_{i} f)^{T}:=\vec{\nabla}f,
	    	\end{split}
		\end{equation}
		and $A(s)$ is the matrix defined in Equation~\eqref{eq:matrix}~with elements 
		\begin{align*}
			A_{ij}(s)=\langle  b^{\prime}\left(X\left( s\right)\right)e_{l_{j}},e_{l_{i}}\rangle_{L^{2}}-\frac{(l_{i}\pi)^{2}}{4M^{2}}K\tilde{\delta}_{ij}+2K\tilde{\delta}_{ij},1\leq i,j\leq n.
		\end{align*}
		On the other hand, by definition
		\begin{align*}
			E[D_{\tilde{h}_{l_{j}}}F]&
			=\sum_{i=1}^{n}E[\int_{0}^{T}\dot{h}_{l_{j}}(s)\sin(l_{j}\pi x_{i})\nabla_{i} f ds],1\leq j\leq n.
		\end{align*}
		Then, substituting Equation~\eqref{eq:3}~into Equation~\eqref{eq:martin}, we obtain the following equation in the vector form
		\begin{align*}
			E[\int_{0}^{T}\langle \vec{\dot{h}}(s),\vec{D}^{F}(s)\rangle_{\mathbb{R}^{n}} ds]
			=&E\left[\int_{0}^{T}\langle \vec{\dot{h}}(s)- A(s)\int_{0}^{s}\vec{\dot{h}}(\tau) d\tau,\vec{\hat{H}}^{F}(s)\rangle_{\mathbb{R}^{n}}ds\right] \\
			=&E\left[\int_{0}^{T}\langle(\mathcal{A}\vec{\dot{h}})(s),\vec{\hat{H}}^{F}(s)\rangle_{\mathbb{R}^{n}} ds\right]=E\left[\int_{0}^{T}\langle\vec{\dot{h}}(s),(\mathcal{A}^{\star}\vec{\hat{H}}^{F})(s)\rangle_{\mathbb{R}^{n}} ds\right].
		\end{align*}
		where $\mathcal{A}^{\star}$ is the dual operator of  $\mathcal{A}$ in $L_{a,n}^{2}$.
		Because $\vec{\dot{h}}(s)\in L_{a,n}^{2}$ is arbitrary, we have	
		\begin{align*}
			E\left[(\mathcal{A}^{\star}\vec{\hat{H}}^{F})(s)|\mathcal{F}_{s}\right]=(\mathcal{A}^{\star}\vec{\hat{H}}^{F})(s)
			=& E[\vec{D}^{F}(s)|\mathcal{F}_{s}],
		\end{align*}
		and 
		\begin{align*}	
			(\mathcal{A}^{\star}k)(t)=k(t)-E\left[\int_{t}^{T}A^{*}(s)k(s)ds|\mathcal{F}_{t}\right].
		\end{align*}			
		Finally,
		\begin{align*}
			\vec{H}^{F}(s)
			=&\mathbf{Q}^{-1}\vec{\hat{H}}^{F}(s)=\mathbf{Q}^{-1}(\mathcal{A}^{\star})^{-1} E[\vec{D}^{F}(s)|\mathcal{F}_{s}].
		\end{align*}
	\end{proof}
	\subsection{The log-Sobolev inequality}\label{sec:4.2}
	In this section, we shall prove the log-Sobolev inequality and the Poincar\'e inequality up to the terminal time $T$ using the Clark-Ocone-Haussmann formula derived in Section \ref{sec:4.1}. The constant in the log-Sobolev inequality will be used to determine the renormalization relation below.
	\begin{theorem}\label{thm:logso}
		For the cylindrical function $F\in\text{Cyl}(\mathcal{L})$  defined in~\eqref{eq:cyl}, we have the following log-Sobolev inequality
		\begin{align}\label{eq:Logsob}
			E[F^{2}\log F^{2}]-E[F^{2}]\log E[F^{2}]
			\leq&2\frac{e^{(1-\gamma-\varepsilon+4K)T}-1}{1-\gamma-\varepsilon+4K}\frac{n^{2}}{\hat{c}(\rho)}E\left[\sum_{i=1}^{n}\sum_{j=1}^{n}\mathcal{K}(x_{i},x_{j})\nabla_{i}f\nabla_{j}f\right],
		\end{align}
where $\hat{c}(\rho)$ is defined in Lemma \ref{lem:lowerbound}.
	\end{theorem}
	\begin{proof}
		Let $\phi=F^{2}$ and $\phi_{t}=E[\phi|\mathcal{F}_{t}]$. By Theorem \ref{thm:finiteresp}, we have
		\begin{align*}
			\phi_{t}=E[\phi]+\sum_{i=1}^{n}\int_{0}^{t}  H^{\phi}(x_{l_{i}},s)dB^{l_{i}}(s).
		\end{align*} 
		By It\^{o} formula,
		\begin{equation}\label{eq:ito}
			\begin{split}
				d\phi_{t}\log(\phi_{t})=&(1+\log(\phi_{t}))d\phi_{t}+\frac{1}{2}\frac{d\left \langle \phi \right \rangle_{t}}{\phi_{t}}\\
				=&(1+\log(\phi_{t}))\left(\sum_{k=1}^{n} H^{\phi}(x_{l_{k}},s)dB^{l_{k}}(t)\right)+\frac{1}{2\phi_{t}}\sum_{k=1}^{n}|H^{\phi}(x_{l_{k}},s)|^{2}dt.
			\end{split}
		\end{equation}
		Note that $\nabla\phi=\nabla F^{2}=2F\nabla F$. By Theorem \ref{thm:finiteresp}, we know that
		\begin{align*}
			\vec{H}^{\phi}(s)
			=&2\mathbf{Q}^{-1}(\mathcal{A}^{\star})^{-1} E[F\vec{D}^{F}(s)|\mathcal{F}_{s}].
		\end{align*}
		Then, by Equation~\eqref{eq:inveroperator}~we have
		\begin{align*}
			(\mathcal{A}^{\star})^{-1}E[F\vec{D}^{F}(s)|\mathcal{F}_{s}] =&E[F\vec{D}^{F}(s)|\mathcal{F}_{s}]+E\left[(\hat{M}^{*})^{-1}(s)\int_{s}^{T}\hat{M}^{*}(\tau)A(\tau)E[F\vec{D}^{F}(\tau)|\mathcal{F}_{\tau}]d\tau|\mathcal{F}_{s}\right]\\
			=&E[F\vec{D}^{F}(s)|\mathcal{F}_{s}]+E\left[F(\hat{M}^{*})^{-1}(s)\int_{s}^{T}\hat{M}^{*}(\tau)A(\tau)\vec{D}^{F}(\tau)d\tau|\mathcal{F}_{s}\right]\\
			=&E[F\vec{\nabla}f|\mathcal{F}_{s}]+E\left[F(\hat{M}^{*})^{-1}(s)\int_{s}^{T}\frac{d}{d\tau}\hat{M}^{*}(\tau)d\tau\vec{\nabla}f|\mathcal{F}_{s}\right]\\
			=&E\left[F(\hat{M}^{*})^{-1}(s)\hat{M}^{*}(T)\vec{\nabla}f|\mathcal{F}_{s}\right],
		\end{align*}
		and
		\begin{align*}
			\Vert\mathbf{Q}^{-1}(\mathcal{A}^{\star})^{-1}E[F\vec{D}^{F}(s)|\mathcal{F}_{s}]\Vert^{2}_{\mathbb{R}^{n}} 
			\leq&E[F^{2}|\mathcal{F}_{s}]E\left[\Vert\mathbf{Q}^{-1}(\hat{M}^{*})^{-1}(s)\hat{M}^{*}(T)\vec{\nabla}f\Vert^{2}_{\mathbb{R}^{n}}|\mathcal{F}_{s}\right].
		\end{align*}
		Therefore, we know that
		\begin{align*}
			\frac{\Vert \vec{H}^{\phi}(s)\Vert^{2}_{\mathbb{R}^{n}}}{E[F^{2}\mid \mathcal{F}_{s}]}=\frac{\Vert \vec{H}^{F^{2}}(s)\Vert^{2}_{\mathbb{R}^{n}}}{E[F^{2}\mid \mathcal{F}_{s}]}\leq& \frac{2\Vert\mathbf{Q}^{-1}(\mathcal{A}^{\star})^{-1} E[F\vec{D}^{F}(s)|\mathcal{F}_{s}]\Vert^{2}_{\mathbb{R}^{n}}}{E[F^{2}\mid \mathcal{F}_{s}]}
			\leq2E\left[\Vert\mathbf{Q}^{-1}(\hat{M}^{*})^{-1}(s)\hat{M}^{*}(T)\vec{\nabla}f\Vert^{2}_{\mathbb{R}^{n}}|\mathcal{F}_{s}\right].
		\end{align*}
		Recall the estimate~\eqref{eq:opestimat}~in Lemma \ref{lem:opera} and $\mathbf{Q}^{-1}\leq \text{I}_{d}$, we get
		\begin{align*}
			&E\left[\int_{0}^{T\wedge\tau_{N}}E\left[\Vert\mathbf{Q}^{-1}(\hat{M}^{*})^{-1}(s)\hat{M}^{*}(T)\vec{\nabla}f\Vert^{2}_{\mathbb{R}^{n}}|\mathcal{F}_{s}\right]ds\right]\\
			=&	E\left[\int_{0}^{T}\textbf{I}_{\{s\leq T\wedge\tau_{N}\}}E\left[\Vert\mathbf{Q}^{-1}(\hat{M}^{*})^{-1}(s)\hat{M}^{*}(T)\vec{\nabla}f\Vert^{2}_{\mathbb{R}^{n}}|\mathcal{F}_{s}\right]ds\right]\\		
			=&	E\left[\int_{0}^{ T\wedge\tau_{N}}\Vert\mathbf{Q}^{-1}(\hat{M}^{*})^{-1}(s)\hat{M}^{*}(T)\vec{\nabla}f\Vert^{2}_{\mathbb{R}^{n}}ds\right]\\	
			\leq&E\left[\int_{0}^{T\wedge\tau_{N}}\Vert\mathbf{Q}^{-1}(\hat{M}^{*})^{-1}(s)\hat{M}^{*}(T\wedge\tau_{N})\vec{\nabla}f\Vert^{2}_{\mathbb{R}^{n}}ds\right]\\
			+&E\left[\int_{0}^{T\wedge\tau_{N}}\Vert\mathbf{Q}^{-1}(\hat{M}^{*})^{-1}(s)\left(\hat{M}^{*}(T)-\hat{M}^{*}(T\wedge\tau_{n})\right)\vec{\nabla}f\Vert^{2}_{\mathbb{R}^{n}}ds\right]\\
			\leq&E\left[\int_{0}^{T\wedge\tau_{N}}\Vert J^{*}(s)\vec{\nabla}f\Vert^{2}_{\mathbb{R}^{n}}ds\right]+E\left[\int_{0}^{T\wedge\tau_{n}}\Vert(\hat{M}^{*})^{-1}(s)\int_{T\wedge\tau_{N}}^{T}\frac{d}{dt}\hat{M}^{*}(t)dt\vec{\nabla}f\Vert^{2}_{\mathbb{R}^{n}}ds\right]\\
			\leq&E\left[\int_{0}^{T\wedge\tau_{N}}e^{(1-\gamma-\varepsilon+4K)(T\wedge\tau_{N}-s)}\Vert \vec{\nabla}f\Vert^{2}_{\mathbb{R}^{n}}ds\right]+E\left[\int_{0}^{T}\Vert(\hat{M}^{*})^{-1}(s)\int_{T\wedge\tau_{N}}^{T}\hat{M}^{*}(t)A^{*}(t)dt\vec{\nabla}f\Vert^{2}_{\mathbb{R}^{n}}ds\right]\\
			\leq&E\left[\frac{e^{(1-\gamma-\varepsilon+4K)(T\wedge\tau_{N})}-1}{1-\gamma-\varepsilon+4K}\Vert \vec{\nabla}f\Vert^{2}_{\mathbb{R}^{n}}\right]+E\left[\int_{0}^{T}\Vert(\hat{M}^{*})^{-1}(s)\int_{T\wedge\tau_{N}}^{T}\hat{M}^{*}(t)A^{*}(t)dt\vec{\nabla}f\Vert^{2}_{\mathbb{R}^{n}}ds\right].
		\end{align*}
		Integrating both sides of~\eqref{eq:ito}~ from 0 to $T\wedge\tau_{N}$ and taking expectations, we have
		\begin{align*}
			E[\phi_{T\wedge\tau_{N}}\log\phi_{T\wedge\tau_{N}}]-\phi_{0}\log\phi_{0}&=\frac{1}{2}E[\int_{0}^{T\wedge\tau_{N}}\frac{\Vert \vec{H}^{F^{2}}(s)\Vert^{2}_{\mathbb{R}^{n}}}{E[F^{2}\mid \mathcal{F}_{s}]}dt]\\
			&\leq E\left[\frac{e^{(1-\gamma-\varepsilon+4K)(T\wedge\tau_{N})}-1}{1-\gamma-\varepsilon+4K}\Vert \vec{\nabla}f\Vert^{2}_{\mathbb{R}^{n}}\right]\\
			&+E\left[\int_{0}^{T}\Vert(\hat{M}^{*})^{-1}(s)\int_{T\wedge\tau_{N}}^{T}\hat{M}^{*}(t)A^{*}(t)dt\vec{\nabla}f\Vert^{2}_{\mathbb{R}^{n}}ds\right].
		\end{align*}
		Thanks to
		\begin{align*}
			\Vert \vec{\nabla}f\Vert^{2}_{\mathbb{R}^{n}}=&2 \sum_{k=1}^{n}\left(\sum_{i=1}^{n}\sin(l_{k}\pi x_{i})\nabla_{i} f\right)^{2}
			\leq\frac{2n^{2}}{\hat{c}(\rho)}\sum_{i=1}^{n}\sum_{j=1}^{n}\mathcal{K}(x_{i},x_{j})\nabla_{i} f\nabla_{j} f,
		\end{align*}
		where $\hat{c}(\rho)$ satisfies~\eqref{eq:lowerbound}.
	Let $N\rightarrow+\infty$, then $\tau_{N}\rightarrow T$ and
		\begin{align*}
			E[\phi_{T}\log\phi_{T}]-\phi_{0}\log\phi_{0}=&E[F^{2}\log F^{2}]-E[F^{2}]\log E[F^{2}]\\
			\leq&2\frac{e^{(1-\gamma-\varepsilon+4K)T}-1}{1-\gamma-\varepsilon+4K}\frac{n^{2}}{\hat{c}(\rho)}E\left[\sum_{i=1}^{n}\sum_{j=1}^{n}\mathcal{K}(x_{i},x_{j})\nabla_{i}f\nabla_{j}f\right].
		\end{align*}

	\end{proof}
	
	The log-Sobolev inequality in Theorem \ref{thm:logso} directly implies the following Poincar\'e inequality, which will be used to estimate the behavior of correlation functions in Section \ref{sec:5}.
	\begin{corollary}\label{cor:poinc}
	Consider the cylindrical function $F\in\text{Cyl}(\mathcal{L})$  defined in~\eqref{eq:cyl}, we get the following Poincar\'e inequality
	\begin{align}\label{eq:poin}
		\textbf{Var}(F(X))\leq& \frac{e^{(1-\gamma-\varepsilon+4K)T}-1}{1-\gamma-\varepsilon+4K}\frac{n^{2}}{\hat{c}(\rho)}E\left[\sum_{i=1}^{n}\sum_{j=1}^{n}\mathcal{K}(x_{i},x_{j})\nabla_{i}f\nabla_{j}f\right].
	\end{align}
\end{corollary}
\begin{proof}
	According to Proposition 5.1.3 in \cite{Bakry2014document}, the log-Sobolev inequality in Theorem \ref{thm:logso} implies the following Poincar\'e inequality 
\begin{align*}
	\textbf{Var}(F(X))\leq& \frac{e^{(1-\gamma-\varepsilon+4K)T}-1}{1-\gamma-\varepsilon+4K}\frac{n^{2}}{\hat{c}(\rho)}E\left[\sum_{i=1}^{n}\sum_{j=1}^{n}\mathcal{K}(x_{i},x_{j})\nabla_{i}f\nabla_{j}f\right].
\end{align*}
\end{proof}
From the result in Theorem \ref{thm:logso}, the constant in the log-Sobolev inequality depends on the parameters $\varepsilon,\rho,K,\gamma$ and the terminal time $T$. The renormalization relation is determined based on the  principle that the constant in the log-Sobolev inequality remains finite as $T\rightarrow+\infty$.
	\begin{theorem}\label{thm:4.8}
		For the cylindrical  function $F\in\text{Cyl}(\mathcal{L})$  defined in~\eqref{eq:cyl}, based on the principle that the constant in~\eqref{eq:Logsob}~remains finite as $T\rightarrow+\infty$, the renormalization relation can be determined as follows
		\begin{align}\label{eq:ronrel}
			K(T)=\frac{1}{T^{\kappa}},\kappa>0,\varepsilon(T)=\frac{1}{T},\delta(T)=\frac{1}{T},\gamma=n^{2}\gamma^{*},\gamma^{*}>1, \rho(T)=\frac{1}{T},
		\end{align}
		then we obtain
		\begin{align*}
			E[F^{2}\log F^{2}]-E[F^{2}]\log E[F^{2}]
			&\leq C(T) \sum_{j=1}^nE\left[|\nabla_j f|^{2} \right],
		\end{align*}
				and
		\begin{align*}
				 \lim\limits_{T\rightarrow+\infty}C(T)=\frac{2}{\gamma^{*}-1}<+\infty,
		\end{align*}
		where
		\begin{align*}
			 C(T):=&2\frac{Tn^{2}\left(e^{T-n^{2}\gamma^{*}T-1+4T^{1-\kappa}}-1\right)}{n^{2}T-n^{2}\gamma^{*}T-1+4T^{1-\kappa}}\frac{T+\sqrt{8\pi} \max_{i}\{\hat{C}_{x_{i}}\}}{T-\sqrt{8\pi} \max_{i}\{\hat{C}_{x_{i}}\}},
		\end{align*}
		and $\hat{C}_{x_{i}}$ will be defined in~\eqref{eq:hatc}~below. 
	\end{theorem}
	\begin{proof}
		Recalling that $\hat{\mathcal{K}}$ is an $n\times n$ matrix with elements $\mathcal{K}(x_{i},x_{j}),1\leq i,j\leq n$, and using the estimates in~\eqref{eq:color}, we know that
		\begin{align*}
		\vert\mathcal{K}(x_{i},x_{j}) \vert\leq
		\frac{1}{\sqrt{8\pi\rho}}e^{-\frac{(x_{j}-x_{i})^{2}}{8\rho}}+\frac{\sqrt{2}\pi^{\frac{3}{2}}\rho^{\frac{1}{2}}}{6},x_{i}\neq x_{j}\in(0,1).
		\end{align*}
		For $i=j$, by Equation~\eqref{eq:cova}~
		\begin{align*}
			\vert\sqrt{8\pi\rho}\mathcal{K}(x_{i},x_{i})-1-e^{-\frac{x_{i}^{2}}{2\rho}}-e^{-\frac{(x_{i}-1)^{2}}{2\rho}}\vert=\sum_{k=1}^{+\infty}\left[e^{-\frac{k^{2}}{2\rho}}-e^{-\frac{(k+x_{i})^{2}}{2\rho}}\right]+\sum_{k=-1}^{+\infty}\left[e^{-\frac{k^{2}}{2\rho}}-e^{-\frac{(k-1+x_{i})^{2}}{2\rho}}\right].
		\end{align*}	
		Then
			\begin{align*}
			\vert\sqrt{8\pi\rho}\mathcal{K}(x_{i},x_{i})-1\vert=&e^{-\frac{x_{i}^{2}}{2\rho}}+e^{-\frac{(x_{i}-1)^{2}}{2\rho}}+2\sum_{k=1}^{+\infty}e^{-\frac{k^{2}}{2\rho}}\\
			\leq& e^{-\frac{x_{i}^{2}}{2\rho}}+e^{-\frac{(x_{i}-1)^{2}}{2\rho}}+2\sum_{k=1}^{+\infty}\frac{2\rho}{k^{2}}
			\leq2\rho(\frac{1}{x_{i}^{2}}+\frac{1}{(x_{i}-1)^{2}})+\frac{2}{3}\pi^{2}\rho.
		\end{align*}	
		Let $C_{x_{i}}:=2(\frac{1}{x_{i}^{2}}+\frac{1}{(x_{i}-1)^{2}})+\frac{2}{3}\pi^{2}$,
		we get
		\begin{align}\label{eq:cx}
			\frac{1}{\sqrt{8\pi\rho}}-\frac{C_{x_{i}}}{\sqrt{8\pi}}\sqrt{\rho}\leq\mathcal{K}(x_{i},x_{i})\leq&\frac{1}{\sqrt{8\pi\rho}}+\frac{C_{x_{i}}}{\sqrt{8\pi}}\sqrt{\rho}.
		\end{align}	
		Let
			\begin{align*}
		\mathcal{S}^{+}_{i}:=\mathcal{K}(x_{i},x_{i})+\sum_{j\neq i}\vert\mathcal{K}(x_{i},x_{j}) \vert,\mathcal{S}^{-}_{i}:=\mathcal{K}(x_{i},x_{i})-\sum_{j\neq i}\vert\mathcal{K}(x_{i},x_{j}) \vert.
		\end{align*}	
		Then for $i=1,\cdots,n,x_{i}\in(0,1)$
			\begin{align*}
			\mathcal{S}^{+}_{i}\leq&\frac{1}{\sqrt{8\pi\rho}}+\frac{C_{x_{i}}}{\sqrt{8\pi}}\sqrt{\rho}+\frac{1}{\sqrt{8\pi\rho}}\sum_{j\neq i}e^{-\frac{(x_{j}-x_{i})^{2}}{8\rho}}+\frac{\sqrt{2}\pi^{\frac{3}{2}}}{6}(n-1)\sqrt{\rho}\\
			\leq&\frac{1}{\sqrt{8\pi\rho}}+\frac{C_{x_{i}}}{\sqrt{8\pi}}\sqrt{\rho}+\sum_{j\neq i}\frac{\sqrt{\frac{8}{\pi}}}{(x_{j}-x_{i})^{2}}\sqrt{\rho}+\frac{\sqrt{2}\pi^{\frac{3}{2}}}{6}(n-1)\sqrt{\rho}=\frac{1}{\sqrt{8\pi\rho}}+\hat{C}_{x_{i}}\sqrt{\rho},\\
			\mathcal{S}^{-}_{i}\geq&\frac{1}{\sqrt{8\pi\rho}}-\frac{C_{x_{i}}}{\sqrt{8\pi}}\sqrt{\rho}-\frac{1}{\sqrt{8\pi\rho}}\sum_{j\neq i}e^{-\frac{(x_{j}-x_{i})^{2}}{8\rho}}-\frac{\sqrt{2}\pi^{\frac{3}{2}}}{6}(n-1)\sqrt{\rho}\\
			\geq&\frac{1}{\sqrt{8\pi\rho}}-\frac{C_{x_{i}}}{\sqrt{8\pi}}\sqrt{\rho}-\sum_{j\neq i}\frac{\sqrt{\frac{8}{\pi}}}{(x_{j}-x_{i})^{2}}\sqrt{\rho}-\frac{\sqrt{2}\pi^{\frac{3}{2}}}{6}(n-1)\sqrt{\rho}=\frac{1}{\sqrt{8\pi\rho}}-\hat{C}_{x_{i}}\sqrt{\rho},
		\end{align*}	
		where
		\begin{align}\label{eq:hatc}
			\hat{C}_{x_{i}}:=\frac{C_{x_{i}}}{\sqrt{8\pi}}+\sum_{j\neq i}\frac{\sqrt{\frac{8}{\pi}}}{(x_{j}-x_{i})^{2}}+\frac{\sqrt{2}\pi^{\frac{3}{2}}}{6}(n-1).
		\end{align}	
		By~\eqref{eq:lowerbound}~, we have
			\begin{align}\label{eq:uplower}
		\frac{1}{\sqrt{8\pi\rho}}\text{I}_{d}- \max_{i}\{\hat{C}_{x_{i}}\}\sqrt{\rho}\text{I}_{d}	\leq\hat{c}(\rho)\text{I}_{d}\leq\hat{\mathcal{K}}\leq\frac{1}{\sqrt{8\pi\rho}}\text{I}_{d}+ \max_{i}\{\hat{C}_{x_{i}}\}\sqrt{\rho}\text{I}_{d}.
		\end{align}
		Then 
			\begin{align*}
			\text{I}_{d}\leq\frac{\hat{\mathcal{K}}}{\hat{c}(\rho)}\leq\frac{1+\sqrt{8\pi} \max_{i}\{\hat{C}_{x_{i}}\}\rho}{1-\sqrt{8\pi} \max_{i}\{\hat{C}_{x_{i}}\}\rho}\text{I}_{d},
		\end{align*}
in the order of the non-negative definiteness. According to the log-Sobolev inequality~\eqref{eq:Logsob}~obtained in Theorem \ref{thm:logso}, we choose the following renormalization relation between the parameters and time $T$
	\begin{align*}
			K(T)=\frac{1}{T^{\kappa}},\kappa>0,\varepsilon(T)=\frac{1}{T},\delta(T)=\frac{1}{T},\gamma=n^{2}\gamma^{*},\gamma^{*}>1, \rho(T)=\frac{1}{T}.
	\end{align*}
	Then
	\begin{align*}
	E[F^{2}\log F^{2}]-E[F^{2}]\log E[F^{2}]
	\leq&2\frac{Tn^{2}\left(e^{T-n^{2}\gamma^{*}T-1+4T^{1-\kappa}}-1\right)}{T-n^{2}\gamma^{*}T-1+4T^{1-\kappa}}\frac{T+\sqrt{8\pi} \max_{i}\{\hat{C}_{x_{i}}\}}{T-\sqrt{8\pi} \max_{i}\{\hat{C}_{x_{i}}\}}\sum_{j=1}^nE\left[|\nabla_j f|^{2} \right]\\
	\leq&2\frac{Tn^{2}\left(e^{T-n^{2}\gamma^{*}T-1+4T^{1-\kappa}}-1\right)}{n^{2}T-n^{2}\gamma^{*}T-1+4T^{1-\kappa}}\frac{T+\sqrt{8\pi} \max_{i}\{\hat{C}_{x_{i}}\}}{T-\sqrt{8\pi} \max_{i}\{\hat{C}_{x_{i}}\}}\sum_{j=1}^nE\left[|\nabla_j f|^{2}\right]\\
	=&C(T) \sum_{j=1}^nE\left[|\nabla_j f|^{2} \right],
\end{align*}
and 
\begin{align*}
	\lim\limits_{T\rightarrow+\infty}C(T)=\frac{2}{\gamma^{*}-1}.
\end{align*}
    \end{proof}	
\section{The partition function and correlation functions}\label{sec:5}
In this section, we first prove that the partition function remains invariant under the renormalization relation obtained in Section \ref{sec:4}, a feature consistent with the behavior observed in the renormalization procedure of the 1D Ising model. We then investigate the long time behavior of correlation functions via the Poincaré inequality, proving that the two point
correlation functions of SPDE on lattices converge as $T\rightarrow +\infty$. The limits coincide with the two point correlation functions of the 1D Ising model at the stable fixed point of the RG transformation. Note that the solution to SPDE~\eqref{eq:2.4}~actually depends on parameters $K,\gamma,\varepsilon,\delta,\rho$, for notation's simplicity, we also write $X(x,t):=X(x,t;K,\gamma,\varepsilon,\delta,\rho)$ when these parameters are fixed. For reader’s convenience, when the parameters satisfy the renormalization relation, we denote it as $\hat{X}(x,T):=X(x,T;K(T),\gamma(T),\varepsilon(T),\delta(T),\rho(T))$.
\subsection{The correctness of the renormalization procedure based on the partition function}\label{sec:6.1}
 Given a finite subset $[N]=\{1,\cdots,N\}\subset\mathbb{Z}$, where $ N$ is a multiple of 3. For the 1D Ising model with nearest-neighbor interactions under periodic boundary conditions, the Hamiltonian takes the following form
 \begin{align}\label{eq:ising}
 \mathcal{H}(\sigma)=-K\sum_{i=1}^{N}\sigma_{i}\sigma_{i+1}+\frac{\gamma}{2}\sum_{i=1}^{N}\sigma^{2}_{i},  \sigma=(\sigma_{1},\cdots,\sigma_{N})\in\Omega:=\{+1,-1\}^{N},
 \end{align}	
 where $\sigma_{1}=\sigma_{N+1}$ and $K,\gamma$ are the coupling parameters of the system. In the renormalization procedure, the partition functions of the system at different scales satisfy the following equation
		\begin{equation}\label{eq:parequ}
			\begin{split}
		\mathcal{Z}(N)=&\sum_{\sigma_{1}=\pm1}\cdots\sum_{\sigma_{N}=\pm1}e^{ K\sum_{i=1}^{N}\sigma_{i}\sigma_{i+1}-\frac{\gamma}{2}\sum_{i=1}^{N}\sigma_{i}^{2}}\\
		=&\sum_{S_{1}=\pm1}\cdots\sum_{S_{N/3}=\pm1}e^{G_{N}+ K_{1}\sum_{i=1}^{N/3}S_{i}S_{i+1}-\frac{\gamma_{1}}{2}\sum_{i=1}^{N/3}S_{i}^{2}},
		\end{split}
	\end{equation}
	where $K_{1},\gamma_{1}$ are the coupling parameters of the system at the blocked scale and $G_{N}$ is a constant such that Equation~\eqref{eq:parequ}~holds. The first line of Equation~\eqref{eq:parequ}~corresponds to the partition function at the original scale, and the second line corresponds to the partition function at the blocked scale after the renormalization procedure. As Cardy noted, the partition functions of the system at the original scale and the blocked scale are the same during the renormalization procedure, see the first paragraph on p.33 in \cite{Cardy1996book} for more details. Denote by $(\partial f)_{i}=f_{i}-f_{i+1}$ the discrete gradient of function $f:\mathbb{Z}\rightarrow \mathbb{R}$, the Hamiltonian~\eqref{eq:ising}~can be rewritten as
	\begin{equation}\label{eq:discreteham}
		\begin{split}
			\mathcal{H}(\sigma)
			=&\frac{1}{2}K\sum_{i=1}^{N}\left(\sigma_{i}-\sigma_{i+1}\right)^{2}+\frac{\gamma-2K}{2}\sum_{i=1}^{N}\sigma^{2}_{i}
			=\frac{1}{2}K\sum_{i=1}^{N}\left|(\partial \sigma)_{i}\right|^{2}+\frac{\gamma-2K}{2}\sum_{i=1}^{N}\sigma^{2}_{i}.
		\end{split}
	\end{equation} 
	The continuous version corresponding to Hamiltonian~\eqref{eq:discreteham}~has the following form
	\begin{align*}
		\mathcal{S}(X)=\frac{K}{2}\int_{0}^{1} \left|\partial_{x} X(x)\right|^{2}dx+\frac{\gamma-2K}{2}\int_{0}^{1}X^{2}(x)dx.
	\end{align*}
	Regarding the time $T$ as the renormalization steps, informally, the partition function can be defined as
\begin{equation}\label{eq:conparti}
	\begin{split}
	\mathcal{Z}(T):
	=&E\left[e^{G_{T}-\frac{\hat{K}(T)}{2}\Vert \partial_{x} X(T)\Vert _{L^{2}}^{2}-\frac{\hat{\gamma}(T)-2\hat{K}(T)}{2}\Vert X(T)\Vert _{L^{2}}^{2}}\right],
	\end{split}
\end{equation}
where $X(x,T)$ is the solution to SPDE~\eqref{eq:2.4}~when the parameters are fixed and $\hat{K}(T),\hat{\gamma}(T)$ are the coupling parameters of the system at time $T$. Now, we shall need to verify that the partition function $\mathcal{Z}(T)$ in Equation~\eqref{eq:conparti}~remains invariant as the solution of SPDE evolves over time, which proves the correctness of the renormalization procedure based on the partition function.
	\begin{theorem}\label{thm:6.1}
		Let $\hat{X}(x,T):=X(x,T;K(T),\gamma(T),\varepsilon(T),\delta(T),\rho(T))$ be the solution to SPDE~\eqref{eq:2.4}~under the renormalization relation~\eqref{eq:ronrel}~selected in Theorem \ref{thm:4.8}. If we take the coupling parameters $\hat{K}(T)=K(T)$ and $\hat{\gamma}(T)=\gamma(T)$, where $K(T)$ and $\gamma(T)$ satisfy the renormalization relation~\eqref{eq:ronrel}, then there exists a function $G_{T}$ such that the partition function 
		\begin{align*}
			\mathcal{Z}(T)=E[e^{G_{T}-\frac{K(T)}{2}\Vert \partial_{x} \hat{X}(T)\Vert _{L^{2}}^{2}-\frac{\gamma(T)-2K(T)}{2}\Vert \hat{X}(T)\Vert _{L^{2}}^{2}}],
		\end{align*}
		remains invariant with respect to $T$. 
	\end{theorem}
	\begin{proof}
		Recall that the renormalization relation~\eqref{eq:ronrel}~in Theorem \ref{thm:4.8} is determined in the limit as $T\rightarrow+\infty$. In fact, it is equivalent to choosing 
		 $K=\frac{1}{1+T^{\kappa}},\kappa>0$ in the renormalization relation. Let
		 \begin{align*}
		 	g(T):=E[e^{-\frac{K(T)}{2}\Vert \partial_{x} \hat{X}(T)\Vert _{L^{2}}^{2}-\frac{\gamma(T)-2K(T)}{2}\Vert \hat{X}(T)\Vert _{L^{2}}^{2}}].
		 \end{align*}
		  Then, on the one hand, we have
		\begin{align*}
			g(T)=E[e^{-\frac{K(T)}{2}\Vert \partial_{x} \hat{X}(T)\Vert _{L^{2}}^{2}-\frac{\gamma(T)-2K(T)}{2}\Vert \hat{X}(T)\Vert _{L^{2}}^{2}}]\leq 1.
		\end{align*}
		On the other hand, by Jessen's inequality
		\begin{align*}
			g(T)=E[e^{-\frac{K(T)}{2}\Vert \partial_{x} \hat{X}(T)\Vert _{L^{2}}^{2}-\frac{\gamma(T)-2K(T)}{2}\Vert \hat{X}(T)\Vert _{L^{2}}^{2}}]\geq e^{-\frac{K(T)}{2}E[\Vert \partial_{x} \hat{X}(T)\Vert _{L^{2}}^{2}]-\frac{\gamma(T)-2K(T)}{2}E[\Vert \hat{X}(T)\Vert _{L^{2}}^{2}]}.
		\end{align*}
		Under the renormalization relation~\eqref{eq:ronrel}, and by virtue of the estimates in Theorems \ref{thm:3.2}, \ref{thm:diffestima}, as well as~\eqref{eq:trespt}, we know that 
		\begin{align*}
			E[\Vert \hat{X}(T) \Vert^{2}_{L^{2}}]&\leq\Vert u\Vert _{L^{2}}^{2}e^{C_{1}T}+\frac{Tr(Q_{\rho})}{C_{1}}(e^{C_{1}T}-1)\\
			&\leq\Vert u\Vert _{L^{2}}^{2}e^{2T(1-n^{2}\gamma^{*}+T+\frac{2}{T^{\kappa}+1})}+\frac{T(e^{2T(1-n^{2}\gamma^{*}+T+\frac{2}{T^{\kappa}+1})}-1)}{24(1-n^{2}\gamma^{*}+T+\frac{2}{T^{\kappa}+1})},\\
				E[\Vert\partial_{x} \hat{X}(T)\Vert _{L^{2}}^{2}]
				&\leq \Vert\partial_{x} u\Vert _{L^{2}}^{2}+\Vert Q_{\partial\rho}^{\frac{1}{2}}\Vert^{2}_{2}T\leq\Vert\partial_{x} u\Vert _{L^{2}}^{2}+\frac{T^{3}}{48}.
		\end{align*}
		Then
		\begin{align*}
			g(T)&\geq e^{-\frac{K(T)}{2}E[\Vert \partial_{x}\hat{X}(T)\Vert _{L^{2}}^{2}]-\frac{\gamma(T)-2K(T)}{2}E[\Vert \hat{X}(T)\Vert _{L^{2}}^{2}]}\\
			&\geq e^{-\frac{1}{2T^{\kappa}+2}\left(\Vert\partial_{x} u\Vert _{L^{2}}^{2}+\frac{T^{3}}{48}\right)-\frac{n^{2}\gamma^{*}(T^{\kappa}+1)-2}{2T^{\kappa}+2}\left(\Vert u\Vert _{L^{2}}^{2}e^{2T(1-n^{2}\gamma^{*}+T+\frac{2}{T^{\kappa}+1})}+\frac{Te^{2T(1-n^{2}\gamma^{*}+T+\frac{2}{T^{\kappa}+1})}-T}{24(1-n^{2}\gamma^{*}+T+\frac{2}{T^{\kappa}+1})}\right)}
			>0.
		\end{align*}
		Therefore, we can choose $G_{T}=-\log g(T)\geq 0$ such that the partition function $\mathcal{Z}(T)$ remains invariant.
	\end{proof}
	
	\subsection{The long time behavior of correlation functions}\label{sec:5.2}
	For the 1D Ising model, every initial point with a finite coupling parameter $K$ will flow toward the stable fixed point $K^{*}=0$ of the RG transformation under the renormalization procedure, and the two point correlation functions of the 1D Ising model at $K^{*}$ are equal to 0, see p.551 in \cite{Morandi2001document} for example. In this section, we shall need to investigate the two point correlation functions of SPDE on lattices, and prove that they will converge to the two point correlation functions of the 1D Ising model at the stable fixed point of the RG transformation as $T\rightarrow +\infty$ under the renormalization relation.
	\begin{theorem}\label{thm:6.3}
		The correlation functions satisfy that for any $x_{1},x_{2}\in(0,1)$
		\begin{align*}
			\vert \textbf{Cov}(X(x_{1},T);X(x_{2},T))\vert^{2}\leq & \left(\frac{e^{(1-\gamma-\varepsilon+4K)T}-1}{1-\gamma-\varepsilon+4K}\right)^{2}\frac{n^{4}}{\hat{c}^{2}(\rho)}\mathcal{K}(x_{1},x_{1})\mathcal{K}(x_{2},x_{2}),
		\end{align*}
	where $\hat{c}(\rho)$ is defined in Lemma \ref{lem:lowerbound}.
	\end{theorem}
	
	\begin{proof}
		Obviously, by Cauchy inequality
		\begin{align*}
			\vert\textbf{Cov}(X(x_{1},T);X(x_{2},T))\vert^{2}&\leq\textbf{Var}(X(x_{1},T))\textbf{Var}(X(x_{2},T)).
		\end{align*}
		Hence by taking the cylindrical functions  $F_{i}(X)=X(x_{i},T)$ for $i=1,2$ in Theorem \ref{thm:logso}, and according to Corollary \ref{cor:poinc}, we have the following Poincar\'e inequalities 
		\begin{align*}
			\textbf{Var}(X(x_{i},T))
			\leq&\frac{e^{(1-\gamma-\varepsilon+4K)T}-1}{1-\gamma-\varepsilon+4K}\frac{n^{2}}{\hat{c}(\rho)}\mathcal{K}(x_{i},x_{i}),i=1,2.
		\end{align*}
		Then
			\begin{align*}
			\vert \textbf{Cov}(X(x_{1},T);X(x_{2},T))\vert^{2}\leq & \left(\frac{e^{(1-\gamma-\varepsilon+4K)T}-1}{1-\gamma-\varepsilon+4K}\right)^{2}\frac{n^{4}}{\hat{c}^{2}(\rho)}\mathcal{K}(x_{1},x_{1})\mathcal{K}(x_{2},x_{2}).
		\end{align*}
	\end{proof}
	We have obtained the renormalization relation and the estimates for the correlation functions of SPDE. Now, we return to the lattices within the interval $(-M,M)$ where $M$ will tend to infinity, and investigate the long time behavior of the two point correlation functions of SPDE on lattices as the terminal time $T\rightarrow+\infty$.
	\begin{corollary}
		If we choose the renormalization relation as follows
		\begin{align}\label{eq:renrela}
			K(T)=\frac{1}{T^{\kappa}},\kappa>0,
			\rho(T)=\frac{1}{T},\gamma^{*}(T)=T,
			\varepsilon(T)=\frac{1}{T},\delta(T)=\frac{1}{T},M(T)=T^{\frac{1}{3}}.
		\end{align}
		Let $\hat{Y}(k,T)$ denote the solution to SPDE~\eqref{eq:2.3}~at the lattice site $k\in\mathbb{Z}$, where the parameters satisfy the renormalization relation~\eqref{eq:renrela}~.
		The two point correlation functions on lattices satisfy
		\begin{align*}
			\textbf{Cov}(\hat{Y}(l,T);\hat{Y}(k,T))&\rightarrow 0,\ as\ T\rightarrow +\infty,
		\end{align*}
		where $l<k\in\mathbb{Z}$.
	\end{corollary}
	\begin{remark}
		This result implies that when the coupling parameter $K$ tends to the stable fixed point of the RG transformation of the 1D Ising model, the two point correlation functions of SPDE on lattices also converge to the two point correlation functions of the 1D Ising model at the stable fixed point of the RG transformation as $T\rightarrow+\infty$. Moreover, according to statistical mechanics theory, see the first paragraph of Section 2.3 on p.391 in \cite{Deligne1999document}, the measure induced by SPDE on lattices will converge to the ensemble measure of the 1D Ising model at the stable fixed point of the RG transformation. This demonstrates that using the constant in the log-Sobolev inequality to determine the renormalization relation is an effective method.
	\end{remark}
	\begin{proof}
		For $l<k\in\mathbb{Z}\cap[-M,M]$, let $x_{1}=\frac{l+M}{2M}$ and $x_{2}=\frac{k+M}{2M}$.
		Recall that $\hat{Y}(l,T)=\hat{X}(\frac{l+M}{2M},T)=\hat{X}(x_{1},T)$ and $\hat{Y}(k,T)=\hat{X}(\frac{k+M}{2M},T)=\hat{X}(x_{2},T)$, then
		\begin{align*}
			\vert \textbf{Cov}(\hat{Y}(l,T);\hat{Y}(k,T))\vert=\vert \textbf{Cov}(\hat{X}(x_{1},T);\hat{X}(x_{2},T))\vert,
		\end{align*}
		and the constants $C_{x_{1}},C_{x_{2}},\hat{C}_{x_{1}},\hat{C}_{x_{2}}$ in Theorem \ref{thm:4.8} satisfy 
		\begin{align*}
			C_{x_{1}}=&2\left((\frac{2M}{l+M})^{2}+(\frac{2M}{M-l})^{2}\right)+\frac{2\pi^{2}}{3},\\
			C_{x_{2}}=&2\left((\frac{2M}{k+M})^{2}+(\frac{2M}{M-k})^{2}\right)+\frac{2\pi^{2}}{3},\\
			\hat{C}_{x_{1}}=&\frac{C_{x_{1}}}{\sqrt{8\pi}}+\frac{4M^{2}}{(l-k)^{2}}\sqrt{\frac{8}{\pi}}+\frac{\sqrt{2}\pi^{3/2}}{6},\\
			\hat{C}_{x_{2}}=&\frac{C_{x_{2}}}{\sqrt{8\pi}}+\frac{4M^{2}}{(l-k)^{2}}\sqrt{\frac{8}{\pi}}+\frac{\sqrt{2}\pi^{3/2}}{6}.
		\end{align*}
		Choosing the renormalization relation as follows
		\begin{align*}
			K(T)=\frac{1}{T^{\kappa}},\kappa>0,
			\rho(T)=\frac{1}{T},\gamma^{*}(T)=T,
			\varepsilon(T)=\frac{1}{T},\delta(T)=\frac{1}{T},M(T)=T^{\frac{1}{3}}.
		\end{align*}
		By~\eqref{eq:cx},~\eqref{eq:uplower}~and Theorem \ref{thm:6.3}, we have
		\begin{align*}
			&\lim\limits_{T\rightarrow+\infty}\vert \textbf{Cov}(\hat{Y}(l,T);\hat{Y}(k,T))\vert\\
			&\leq\lim\limits_{T\rightarrow+\infty}\frac{e^{(1-\gamma-\varepsilon+4K)T}-1}{1-\gamma-\varepsilon+4K}\frac{n^{2}}{\hat{c}(\rho)}\sqrt{\mathcal{K}(x_{1},x_{1})\mathcal{K}(x_{2},x_{2})}\\
			&\leq\lim\limits_{T\rightarrow+\infty}\frac{Tn^{2}\left(e^{T-n^{2}T^{2}-1+4T^{1-\kappa}}-1\right)}{T-n^{2}T^{2}-1+4T^{1-\kappa}}\frac{\sqrt{(T+C_{x_{1}})(T+C_{x_{2}})}}{T-\sqrt{8\pi} \max\{\hat{C}_{x_{1}},\hat{C}_{x_{2}}\}}=0.
		\end{align*}

	\end{proof}
	
	\section{Acknowledgements}
	This work was supported by National Natural Science Foundation of China (Grant No.12288201).

\end{document}